\documentclass[11pt]{amsart}

\usepackage{color}
\usepackage{mathrsfs}
\usepackage{url}
\usepackage{amscd}
\usepackage{pictexwd,dcpic}
\usepackage[arrow, matrix, curve]{xy}

\usepackage[american]{babel}
\usepackage[latin1]{inputenc}

\usepackage[neveradjust]{paralist}               	
\usepackage[parfill]{parskip}			
\usepackage{a4wide}
\usepackage{amsmath,amsthm}	
\usepackage{amssymb}
\usepackage[T1]{fontenc}

\usepackage[bottom]{footmisc}			
\usepackage{bbm, dsfont}  
\renewcommand{\1}{\mathbbm{1}}
\newcommand{\imp}{\Rightarrow}
\newcommand{\en}{\enspace}
\newcommand{\sk}[1]{\left\langle #1 \right\rangle}

\def \A {\mathcal{A}}
\def \L {\mathscr{L}}
\def \T {\mathcal{T}}
\def \TT {\mathscr{T}}

\def \S {\mathcal{S}}
\def \G {\mathcal{G}}

\def \R {\mathbb{R}}

\def \N {\mathbb{N}}
\def \C {\mathbb{C}}

\def \e {\varepsilon}

\def \ol {\overline}

\def \a {\alpha}

\def \ph {\varphi}

\def \Torus {\mathbb{T}}

\def \Rep {\operatorname{Rep}}
\def \ran {\operatorname{ran}}
\def \rg {\operatorname{rg}}
\def \lin {\operatorname{lin}}

\def \co {\operatorname{co}}

\def \Fix {\operatorname{Fix}}
\def \supp {\operatorname{supp}}
\def \Re {\operatorname{Re}}
\def \Im {\operatorname{Im}}

\theoremstyle{plain}
\newtheorem*{thm*}{Theorem}
\newtheorem{thm}{Theorem}[section]
\newtheorem{lemma}[thm]{Lemma}
\newtheorem{prop}[thm]{Proposition}
\newtheorem{cor}[thm]{Corollary}
\theoremstyle{definition}
\newtheorem{definition}[thm]{Definition}
\newtheorem{example}[thm]{Example}

\newtheorem{remark}[thm]{Remark}

\title{Topological Wiener-Wintner
  theorems for amenable operator semigroups}
\author{Marco Schreiber}
\address{Institute of Mathematics\\University of T\"ubingen\\Auf der
  Morgenstelle 10\\72076 T\"ubingen\\Germany}
\email{masc@fa.uni-tuebingen.de}
\date\today

\begin{document}		

\maketitle

\begin{abstract}
Inspired by topological Wiener-Wintner theorems we study the mean 
ergodicity of amenable semigroups of Markov operators on $C(K)$ and
show the connection to the convergence of strong and weak ergodic
nets.  
The results are then used to characterize mean ergodicity of Koopman
semigroups corresponding to skew product actions on compact group extensions.
\end{abstract}

\vspace{0.5cm}


Robinson's topological Wiener-Wintner theorem \cite[Theorem
1.1]{robinson94} is concerned with the uniform convergence of the weighted
Cesàro averages
\begin{equation}
  \label{eq:1}
\frac{1}{n}\sum_{j=1}^n \lambda^j S^j f  \tag{$\star$}
\end{equation}
for a continuous function $f\in C(K)$ on a compact space $K$, the
Koopman operator $S: f\mapsto f\circ \ph$ of a continuous 
transformation $\ph: K\to K$ and $\lambda$ in the unit circle $\Torus$. 
Subsequently, Robinson's result has been generalized in various ways
by Walters \cite{walters96}, Santos and Walkden \cite{santos07} and
Lenz \cite{lenz09a, lenz09}. 

It turned out that the uniform convergence of Wiener-Wintner
averages plays an important role in the
mathematical description of diffraction on quasicrystals. 
In \cite{lenz09} Lenz showed how the intensity of Bragg peaks can be
calculated via certain limits of Wiener-Wintner averages, giving a
partial answer to a conjecture of Bombieri and Taylor
\cite{bombieri86, bombieri87}. 

So far, all these authors focused on the convergence of a particular
sequence of Cesàro means similar to (\ref{eq:1}). 
In this paper we take a more general view and look at semigroups of
operators being mean ergodic on $C(K)$ or on some closed invariant
subspace. 
Based on the theory of mean ergodic semigroups (see
\cite[Chapter 2]{krengel85}) this allows us to unify and extend 
the known Wiener-Wintner theorems to amenable semigroups of Markov
(instead of Koopman) operators on $C(K)$.

The problem when averages of the form (\ref{eq:1}) even converge uniformly
in $\lambda\in\Torus$ has been studied independently by Assani
\cite{assani03} and Robinson \cite{robinson94} with their results
subsequently generalized by Walters \cite{walters96}, Santos and
Walkden 
\cite{santos07} and Lenz \cite{lenz09}. 
In \cite{schreiber12} we have developed the concept of a uniform
family of ergodic nets that allows us to treat this question also in
our more general setting.


 In the first part of this paper we study mean ergodicity of
 semigroups of Markov operators on $C(K)$.  
 For an amenable representation $\{S_g: g\in G\}$ of a semitopological
 semigroup $G$ as Markov operators and for $\chi:G\to \Torus$ a 
 continuous multiplicative map, we then characterize mean ergodicity of the
 semigroup $\{\chi(g)S_g: g\in G\}$.  

In the second part we restrict our attention to Koopman operators on
the space $C(K,\C^N)$ of continuous $\C^N$-valued functions and show
similar results replacing $\chi:G\to\Torus$ by a
continuous cocycle $\gamma: G\times K\to U(N)$ into the group of
unitary operators on $\C^N$. 

In the third part we consider skew product actions on compact group
extensions. 
We use the previous results in order to characterize mean ergodicity
of the corresponding Koopman representation. 
Finally, we obtain a new proof and a generalization of a theorem of
Furstenberg, showing that an ergodic skew product action corresponding
to a uniquely ergodic action is uniquely ergodic.

\section{Semigroups of Markov operators}
\label{sec:ww}
We consider the space $C(K)$ of complex valued continuous functions on a compact set
$K$, a semitopological semigroup $G$ (see Berglund et
al. \cite[Chapter 1.3]{berglund89}) and assume that $\S=\{S_g:g\in
G\}$ is a bounded representation of $G$ on $C(K)$, i.e.,
\begin{enumerate}[(i)]
\item $S_g\in \L(C(K))$ for all $g\in G$ and $\sup_{g\in G}\|S_g\|<\infty$,
\item $S_{g_1} S_{g_2}=S_{g_2 g_1}$ for all $g_1, g_2\in G$,
\item $g\mapsto S_g f$ is continuous for all $f\in C(K)$.
\end{enumerate}
Such a bounded representation $\S$ and its
convex hull $\co\S$ are topological semigroups with
respect to the strong operator topology. 

On the
dual space $C(K)'$, identified with the set $M(K)$ of regular Borel measures
on $K$, we consider the adjoint semigroup $\S':=\{S_g': g\in G\}$. 

A \emph{mean} on
the space $C_b(G)$ of bounded continuous functions on $G$ is a linear
functional $m\in C_b(G)'$ satisfying
 $\sk{m,\1}=\|m\|=1$. 
A mean $m\in C_b(G)'$ is called \emph{right (left) invariant} if
$$\sk{m,R_g f}=\sk{m,f} \quad \left(\sk{m,L_g f}=\sk{m,f}\right)
 \quad\forall g\in G, f\in C_b(G),$$ 
where 
$R_g f(h)=f(hg)$ and $L_gf(h)=f(gh)$ for $h\in G$.
A mean $m\in C_b(G)'$ is called \emph{invariant} if it is both right and left invariant.

The semigroup $G$ is called \emph{right (left) amenable} if there
exists a right (left) invariant mean on $C_b(G)$. It is called
\emph{amenable} if there exists an invariant mean on $C_b(G)$ (see Berglund et
al. \cite[Chapter 2.3]{berglund89} or the survey article of Day
\cite{day69}).
%
%
Notice that if $\S:=\{S_g:g\in G\}$ is a bounded representation of a
right (left)
amenable semigroup $G$
on $X$, then $\S$ endowed with the strong 
operator topology is also right (left) amenable.
In the following, the space $\L(C(K))$ will be endowed with
the strong operator topology unless stated otherwise.

 A net $(A_\a^\S )_{\a\in \A}$ of operators in $\L(C(K))$ is called a
\emph{strong right (left) $\S$-ergodic net}  if the following conditions hold.
  \begin{enumerate}
  \item $A_\a^\S \in \ol{\co}\S$ for all $\a\in \A$.

  \item $(A_\a^\S)$ is \emph{strongly right (left) asymptotically $\S$-invariant}, i.e.,

$\lim_\a A_\a^\S f-A_\a^\S S_g f=0\en \left(\lim_\a A_\a^\S f-
S_g A_\a^\S f=0\right)$ for all $f\in C(K)$ and  $g\in G$.
  \end{enumerate}
%
The net $(A_\a^\S )$ is called a \emph{strong $\S$-ergodic net} if it
is a strong right and left $\S$-ergodic net. 
Clearly, the Cesàro means $\frac{1}{n}\sum_{j=1}^{n}S^j$ of a contraction
$S\in \L(C(K))$ form a strong
$\{S^j:j\in\N\}$-ergodic net and we refer to
\cite{eberlein49,sato78,schreiber12} for many more examples.

 The semigroup $\S$ is called \emph{mean ergodic} if $\ol{\co}\S$
 contains a zero element $P$ (see \cite[Chapter~1.1]{berglund89}),
 which is called the \emph{mean ergodic projection of $\S$}. (See e.g.
 Krengel \cite[Chapter~2]{krengel85} for an introduction to this concept.)

Denote by $\Fix\S=\{f\in C(K): S_g f=f\;\forall g\in G\}$ and
$\Fix\S'=\{\nu\in C(K)': S_g' \nu=\nu \;\forall g\in G\}$ the fixed spaces of
$\S$ and $\S'$, respectively, and by $\lin\rg(I-\S)$ the linear span
of the set $\rg(I-\S)=\{f-S_g f: f\in C(K), g\in G\}$.

We recall some characterizations of mean ergodicity from Theorem~1.7
and Corollary~1.8 in \cite{schreiber12}.

\begin{prop}
  \label{prop:mean-ergodic}
Let $G$ be represented on $C(K)$ by a bounded (right) amenable
semigroup $\S=\{S_g:g\in G\}$. Then the following assertions are equivalent.
\begin{enumerate}
\item $\S$ is mean ergodic with mean ergodic projection $P$.
\item $\Fix\S$ separates $\Fix\S'$.
\item $C(K)= \Fix\S\oplus\ol{\lin}\rg(I-\S)$.
\item $A_\a^\S f$ converges weakly (to a fixed point of $\S$) for
  some/every strong (right) $\S$-ergodic net $(A_\a^\S)$ 
and all $f\in C(K)$.
\item $A_\a^\S f$ converges strongly (to a fixed point of $\S$) for
  some/every strong (right) $\S$-ergodic net $(A_\a^\S)$ 
and all $f\in C(K)$.
\end{enumerate}
The limit $P$ of the nets $(A_\a^\S)$ in the weak (strong, resp.) operator
topology is the mean ergodic projection of $\S$ mapping $C(K)$ onto 
$\Fix \S$ along $\ol{\lin}\rg(I-\S)$. 
\end{prop} 

Let now $G$ be represented on $C(K)$ by a semigroup
$\S=\{S_g: g\in G\}$ of \emph{Markov operators}, i.e., of positive operators satisfying $S_g\1=\1$ for all $g\in G$. 
Then $\S$ consists of contractions and hence $\S$ is bounded.
Assume that the semigroup $\S$ is \emph{uniquely ergodic}, i.e.,
$\Fix\S'=\C\cdot \mu$ for some probability measure $\mu\in C(K)'$.
We denote by $S_{g,2}$ the continuous extension of the
operator $S_g\in\S$ to the space $L^2(K,\mu)$.  
The corresponding extended semigroup is $\S_2:=\{S_{g,2}:
g\in G\}$ with $\S_2^*:=\{S_{g,2}^*:g\in G\}$ the
semigroup of Hilbert space adjoints. 
The semigroup $\S$ is called \emph{ergodic with respect to $\mu$} if
$\Fix\S_2=\C\cdot\1$.
Since in $L^2(K,\mu)$ all contraction
semigroups are mean ergodic (see, e.g., \cite[Corollary 1.9]{schreiber12}),
$\S_2$ is mean ergodic. 
In fact, the above
assumptions even imply mean ergodicity on $C(K)$ (cf. Eisner, Farkas, Haase and Nagel \cite[Theorem 10.6]{efhn}
and Krengel \cite[Chapter 5, Section 5.1]{krengel85}
for representations of $\N$).

\begin{prop}
\label{prop:uniquely-ergodic-m-erg}
Let $G$ be represented on $C(K)$ by a right amenable semigroup
$\S=\{S_g:g\in G\}$ of Markov operators. Then (1) implies (2) in the following statements. 
\begin{enumerate}
\item $\S$ is uniquely ergodic.
\item $\S$ is mean ergodic and $\Fix\S=\C\cdot\1$.
\end{enumerate}
If there exists $0<\mu\in \Fix\S'$,
then (2) implies (1). 
\end{prop}
\begin{proof}
(1)$\imp$(2):  Since $\Fix\S$ contains the constant functions, it separates
  $\Fix\S'$ and hence $\S$ is mean ergodic by Proposition
  \ref{prop:mean-ergodic}.
To show $\Fix\S=\C\cdot\1$ it suffices to prove that $\Fix\S'$ separates
  $\Fix\S$. To see this, take $0\neq x\in \Fix\S$ and let $P$ be the
  mean ergodic projection of $\S$. Choose $x'\in X'$ with
  $\sk{x,x'}\neq 0$. Since $Px\in\ol{\co}\S x=\{x\}$ this implies
  $\sk{x,P'x'}=\sk{Px,x'}=\sk{x,x'}\neq 0$ and $P'x'\in\Fix\S'$
  follows by taking adjoints in the equality $PS_g=P$ for all $g\in G$.

Assume now that there exists $0<\mu\in \Fix\S'$.\\
(2)$\imp$(1): If $\Fix\S=\C\cdot\1$ separates $\Fix\S'$, then $\Fix\S'$
can be at most one dimensional. But by hypothesis $\Fix\S'$ is at
least one dimensional and hence $\Fix\S'=\C\cdot\mu$.
\end{proof}

Notice that if in the situation of Proposition
\ref{prop:uniquely-ergodic-m-erg} $\S$ is also left amenable,  then Day's fixed point theorem
\cite[Chapter V, Section 2, Theorem 5]{day73} ensures the existence of
a probability measure $\mu\in \Fix\S'$. This
leads to the following corollary.

\begin{cor}
\label{cor:uniquely-ergodic-m-erg}
  Let $G$ be represented on $C(K)$ by a amenable semigroup
$\S=\{S_g:g\in G\}$ of Markov operators. Then the following assertions
are equivalent.
\begin{enumerate}
\item $\S$ is uniquely ergodic.
\item $\S$ is mean ergodic and $\Fix\S=\C\cdot\1$.
\end{enumerate}
\end{cor}

In the following we will always assume that $G$ is an amenable semigroup.

Let $\widehat{G}$ be the set of all \emph{characters of $G$}, i.e.,
the set of all continuous multiplicative maps 
$\chi:G\to \Torus$ (see \cite{williamson67}), and take 
$\chi\in\widehat{G}$.  
Then we consider the semigroup
$\chi \S:=\{\chi(g)S_g: g\in G\}$ and denote by
$(\chi\S)':=\{(\chi(g)S_g)': g\in G\}$ the adjoint semigroup on $C(K)'$.
Notice that $\chi\S$ is amenable as a bounded representation of the amenable semigroup $G$.

%

Again, $\chi\S$ extends to $L^2(K,\mu)$ and the extended semigroup
$\chi\S_2$ is contractive, hence mean ergodic on $L^2(K,\mu)$. But unlike $\S$, the semigroup
$\chi\S$ is not always mean ergodic on $C(K)$. In \cite[Proposition 3.1]{robinson94} Robinson
gave an elaborate example for such a situation.
%
Here is a much simpler one due to Roland Derndinger (oral communication). 
\begin{example}
\label{ex:roland}
%
Consider the set $\{-1,1\}^\N$ endowed with the product topology and
for $i\in\N$ define the sequence $x^{(i)}=(x^{(i)}_n)_{n\in\N}\in \{-1,1\}^\N$
by 
$$x^{(i)}_n:=\left\{
\begin{array}{ll}
  (-1)^n,& n<i\\
(-1)^{n+1},& n\ge i.
\end{array}
\right . $$
If $\ph$ denotes the left shift on $\{-1,1\}^\N$, i.e.,
$\ph((x_n))=(x_{n+1})$, then the set $K:=\{\pm x^{(i)}: i\in\N\}\subset \{-1,1\}^\N$ is
a closed $\ph$-invariant subset of $\{-1,1\}^\N$.
 Let $S$ be the
corresponding Koopman operator on $C(K)$, i.e., $Sf=f\circ\ph$ for
$f\in C(K)$. 
We claim that the semigroup $\S:=\{S^n: n\in \N\}$ is uniquely ergodic, but if
$\chi\in\widehat{\N}$ is given by $\chi(n)=(-1)^n$, then
$\chi\S=\{(-S)^n: n\in\N\}$ is not mean ergodic.

First, notice that $\Fix \S=\C\cdot\1$. 
Indeed, if $f\in \Fix\S$, then
for all $i\in\N$ there exists $n>i$ such that
$f(\pm x^{(i)})=S^n f(\pm x^{(i)})=f(x^{(1)}),$
and thus $f$ is constant. 
To show that $\S$ is uniquely ergodic it thus suffices by Corollary
\ref{cor:uniquely-ergodic-m-erg} to show that $\S$ is mean ergodic. 
  
Let $f\in C(K)$. Then  $\frac{1}{N}\sum_{n=0}^{N-1}S^nf$ converges
pointwise to the continuous function $\pm x^{(i)}\mapsto
\frac{1}{2}(f(x^{(1)})+f(-x^{(1)}))$ since $f(\ph^n(\pm
x^{(i)}))=f((-1)^{n+i}(\pm x^{(1)}))$ for all $n>i$. 
Since weak and pointwise convergence coincide for bounded sequences,
it follows from Proposition \ref{prop:mean-ergodic} that $\S$ is mean
ergodic. 


We now show that $\chi\S= \{(-S)^n: n\in\N\}$ is not mean ergodic. Let
$f_1\in C(\{-1,1\}^\N)$ be defined by $f_1((x_n))= x_1$ and take its
restriction $f_1$ to $K$. 
Then $\frac{1}{N}\sum_{n=0}^{N-1}(-S)^nf_1$ converges pointwise to the
function $h$ defined by $h(\pm x^{(i)})=\pm 1$ for all $i\in\N$. But $h\notin
C(K)$ since $x^{(i)}\rightarrow -x^{(1)}$ and
$h(x^{(i)})=1\nrightarrow -1=h(-x^{(1)})$. Hence the sequence
$\left(\frac{1}{N}\sum_{n=0}^{N-1}(-S)^nf_1\right)_N$ does not
converge in $C(K)$ and thus $\chi\S$ is not mean ergodic.
\end{example}

Motivated by this example and various papers in mathematical physics
on diffraction theory of quasicrystals and on the Bombieri-Taylor
conjecture (see e.g.  \cite{oliveira98, hof95, lenz09a, lenz09}), we
now characterize the mean ergodicity of the semigroup $\chi\S$.

Let us first recall some facts about the lattice structure of
$C(K)'$ (see \cite[Appendix D.2]{efhn} for details). For a bounded linear
functional $\nu\in C(K)'$ one defines a mapping $|\nu|$ by 
$$\sk{|\nu|,f}:=\sup\{|\sk{\nu, h}|: h\in C(K), |h|\le f\}$$
for $0\le f\in C(K)$ and extends it uniquely to a bounded
linear functional $|\nu|\in C(K)'$. 
With this structure the space $C(K)'$ becomes a Banach
lattice. On the other hand,
the space $M(K)$ of regular Borel measures on $K$ is a Banach
lattice with the total variation $|\nu|$ of a measure $\nu\in M(K)$ defined by
$$|\nu|(E):=\sup\left\{\sum_{j=1}^\infty|\nu(E_j)|: (E_j)_{j\in\N} \text{ a
  partition of } E\right\},\quad (E\subset K\text{ measurable}),$$
and the norm $\|\nu\|:=|\nu|(K)$.
The notation $|\nu|$ for a functional $\nu\in C(K)'$ and a measure
$\nu\in M(K)$ is justified since
the mapping 
\begin{align*}
d:M(K)&\to C(K)'\\
\nu&\mapsto d\nu
\end{align*}
in the Riesz Representation Theorem is a lattice
isomorphism. 

For a function $h\in L^2(K,\mu)$ we denote by $\ol{h} d\mu\in C(K)'$ the
functional defined by 
$$\sk{\ol{h}d\mu,f}:=\sk{f,h}_{L^2(K,\mu)}=\int_K f(x)\ol{h(x)}\,d\mu(x) \quad (f\in C(K)).$$
\begin{lemma}
  \label{lemma:Fixraum_dualer_Fixraum}
Let $G$ be represented on $C(K)$ by a 
semigroup $\S=\{S_g:g\in G\}$ of Markov operators. 
If $\S$ is uniquely
ergodic with invariant measure $\mu$, then for each
$\chi\in\widehat{G}$ the map 
\begin{align*}
 L^2(K,\mu)\supset\Fix(\chi\S_2)^*&\to\Fix (\chi\S)'\subset C(K)'\\
h\quad&\mapsto\quad \ol{h}d\mu
\end{align*}
is antilinear and bijective.
\end{lemma}
\begin{proof}
 To see that the map is well-defined, let $h\in\Fix(\chi\S_2)^*$. For all $f\in C(K)$ and $g\in G$
  we have
 \begin{align*}
  \sk{(\chi(g)S_g)'(\ol{h} d\mu),f}&=\sk{\ol{h} d\mu,\chi(g)S_g f}\\
&=\sk{\chi(g)S_{g,2} f,h}_{L^2(K,\mu)}\\
&=\sk{f,(\chi(g)S_{g,2})^*h}_{L^2(K,\mu)}\\
&=\sk{f,h}_{L^2(K,\mu)}
=\sk{\ol{h} d\mu,f},
  \end{align*}
 yielding $\ol{h}d\mu\in\Fix(\chi\S)'$. 

Since antilinearity and injectivity are clear, it remains to show surjectivity. Let
$\nu\in\Fix(\chi\S)'$. We claim that $|\nu|\le
S_g'|\nu|$ for all $g\in G$. Indeed, if $0\le f\in C(K)$ and $g\in G$,
then
\begin{align*}
  \sk{|\nu|,f}&=\sup_{|\tilde{f}|\le f}|\langle\nu,\tilde{f}\rangle|
=\sup_{|\tilde{f}|\le f}|\langle (\chi(g)S_g)'\nu, \tilde{f}\rangle|\\
&=\sup_{|\tilde{f}|\le f}|\langle\nu,\chi(g)S_g \tilde{f}\rangle|\\
&\le \sup_{|\tilde{f}|\le f}\langle|\nu|,|S_g \tilde{f}|\rangle\\
&\le \sup_{|\tilde{f}|\le f}\langle|\nu|,S_g|\tilde{f}|\rangle\\
&=\sk{|\nu|,S_g f}=\sk{S_g'|\nu|,f}.
\end{align*}
If $0\le f\in
C(K)$ and $g\in G$, then 
$$0\le\sk{S_g'|\nu|-|\nu|,f}\le
\sk{S_g'|\nu|-|\nu|,\|f\|_\infty\1}=\|f\|_\infty\sk{|\nu|,S_g\1-\1}=0.$$ 
Hence $|\nu|\in \Fix\S'=\C\cdot \mu$ by unique ergodicity and thus
$\nu$ is absolutely continuous with respect to $\mu$. 
The Radon-Nikodým Theorem~then implies
  the existence of a function $h\in L^\infty(K,\mu)$  such that 
$\nu=\ol{h} d\mu$. 
The same calculation as above shows that $h\in \Fix(\chi
  \S_2)^*$.
\end{proof}

Note that for a contraction $T$ on a Hilbert space $H$ the fixed spaces of $T$ and its adjoint $T^*$ coincide. 
Indeed, for each $x\in \Fix T$ we have
\begin{align*}
  \|T^*x-x\|^2&=\|T^*x\|^2-\sk{T^*x,x}-\sk{x,T^*x}+\|x\|^2\\
&\le \|x\|^2-\sk{x,Tx}-\sk{Tx,x}+\|x\|^2=0,
\end{align*}
which yields $\Fix T=\Fix T^*$ by symmetry.

Now, if $G$ is represented on $C(K)$ by a semigroup $\S$ of Markov operators,
then $\S_2$ consists of contractions on $L^2(K,\mu)$ and thus the fixed
spaces of $\S_2$ and $\S_2^*$ coincide. Hence, it follows from Lemma
\ref{lemma:Fixraum_dualer_Fixraum} applied to the constant character
$\1\in\widehat{G}$, that unique ergodicity of $\S$ with respect to $\mu$ implies
ergodicity of $\S$ with respect to $\mu$.

\begin{lemma}
  \label{lemma:dimension-fixraum}
Let $G$ be represented on $C(K)$ by a  
semigroup $\S=\{S_g:g\in G\}$ of Markov operators. If $\S$ is
ergodic with respect to some invariant measure $\mu$
and $\chi\in\widehat{G}$, then $\dim\Fix \chi\S_2\le 1$.
\end{lemma}
\begin{proof}
The semigroup $\S_2$ consists of contractions on $L^2(K,\mu)$ and thus
the closure $\TT$ of $\S_2$ with respect to the weak operator topology
contains a unique minimal idempotent $Q$ (cf. \cite[Theorem
16.11]{efhn}). 
By \cite[Theorem~16.22]{efhn} the minimal ideal $\G=\TT Q$ of $\TT$ is
a compact group (even for the strong operator topology) and the
map $T\mapsto T|_{\ran Q}$ from $\G$ to $\{T|_{\ran Q}: T\in\TT\}$ is
a topological isomorphism of compact groups. 
The projection $Q$ is positive since each operator in $\S_2$ is
positive.
Moreover, $Q$ is an orthogonal projection onto its range with
$Q\1=\1$.
Since $\sk{Qf,\1}_{L^2}=\sk{f,\1}_{L^2}>0$ for each $0<f\in L^2(K,\mu)$, $Q$
is strictly positive on $L^2(K,\mu)$. 
Hence $\ran Q$ is a sublattice of $L^2(K,\mu)$ by
\cite[Proposition~11.5]{schaefer74}. 
If $T_g$ denotes the restriction of $S_{g,2}$ to $\ran Q$, then $T_g$ is
invertible with positive inverse, hence $T_g$ is a 
lattice homomorphism with $T_g\1=\1$ for each $g\in G$. 
By \cite[Theorem~7.18]{efhn} each $T_g$ is then an algebra
homomorphism on the subalgebra $\ran Q\cap L^\infty(K,\mu)$. 



Now, take $\chi\in\widehat{G}$. If $f \in \Fix\chi\S_2$, then $f$ generates a one-dimensional $\S_2$-invariant subspace of $L^2(K,\mu)$ and hence by \cite[Theorem~16.29]{efhn} is contained in $\ran Q$.
Since $S_{g,2}$ is a lattice homomorphism on $\ran Q$, we have
$|f|=|\ol{\chi(g)}S_{g_2}f|=S_{g,2}|f|$ for each $g\in G$, hence by ergodicity, $|f|\in\C\cdot\1$. 
So, if $f,h\in \Fix\chi\S_2\setminus\{0\}$ we have $f,h\in\ran Q\cap L^\infty(K,\mu)$ and we may assume $|f|=|h|=\1$. 
We then obtain
$$S_{g,2}(f\cdot \ol{h})=T_g(f\cdot\ol{h})=T_gf\cdot T_g\ol{h}=\ol{\chi(g)}f\cdot
\chi(g)\ol{h}=f\cdot \ol{h}$$
for each $g\in G$. 
Hence $f\cdot\ol{h}\in\Fix\S_2$ and therefore $f\cdot \ol{h}=c\cdot\1$ for some $c\in \C$, which yields $f=c\cdot h$.
\end{proof}

The following theorem is our first main result.

\begin{thm}
  \label{thm:mean-ergodic}
Let $\S=\{S_g:g\in G\}$ be a representation of a (right) amenable
semigroup $G$ as Markov operators on $C(K)$ and assume that $\S$ is
uniquely ergodic with invariant measure $\mu$.
For $\chi\in\widehat{G}$ the following assertions are equivalent.
\begin{enumerate}[(1)]
\item $\Fix\chi\S_2\subseteq \Fix\chi\S$.
\item $\chi\S$ is mean ergodic with mean ergodic projection $P_\chi$.
\item $\Fix\chi\S$ separates $\Fix(\chi\S)'$.
\item $C(K)= \Fix\chi\S\oplus\ol{\lin}\rg(I-\chi\S)$.
\item $A_\a^{\chi\S} f$ converges weakly (to a fixed point of $\chi\S$)
  for some/every strong (right) $\chi\S$-ergodic
  net $(A_\a^{\chi\S})$ and all $f\in C(K)$.
\item $A_\a^{\chi\S} f$ converges strongly (to a fixed point of
  $\chi\S$) for some/every strong (right) $\chi\S$-ergodic net
  $(A_\a^{\chi\S})$ and all $f\in C(K)$. 
\end{enumerate}
The limit $P_\chi$ of the nets $(A_\a^{\chi\S})$ in the strong (weak, resp.)
operator topology is the mean ergodic projection of
$\chi\S$ mapping $C(K)$ onto $\Fix \chi\S$ along
$\ol{\lin}\rg(I-\chi\S)$.
\end{thm} 



\begin{proof}
The equivalence of the statements (2) to (6) follows directly from
Proposition~\ref{prop:mean-ergodic}.

(1)$\imp$(3): If $0\neq\nu\in \Fix(\chi\S)'$, then $\nu=\ol{h}d\mu$ by
Lemma~\ref{lemma:Fixraum_dualer_Fixraum} for some $0\neq h\in
\Fix(\chi\S_2)^*$. Since $\chi\S_2$ consists of contractions on
$L^2(K,\mu)$, we have $\Fix(\chi\S_2)^*=\Fix\chi\S_2$.
Since $\Fix\chi\S_2\subseteq \Fix\chi\S$
by (1), this yields $h\in \Fix\chi\S$ and
$$\sk{\nu,h}=\|h\|^2_{L^2(K,\mu)}>0.$$

(3)$\imp$(1): Suppose $f\in \Fix\chi\S_2\setminus \Fix\chi\S$. Then
$\dim\Fix\chi\S_2=1$ and $\dim\Fix\chi\S=0$ by
Lemma~\ref{lemma:dimension-fixraum}, while  
$\dim\Fix(\chi\S)'=1$ by Lemma~\ref{lemma:Fixraum_dualer_Fixraum}. 
Hence $\Fix\chi\S$ does not separate $\Fix(\chi\S)'$. 
\end{proof}


As Example~\ref{ex:roland} shows, mean ergodicity of
$\chi\S$ does not hold on $C(K)$ in general. 
The following theorem characterizes mean ergodicity of $\chi\S$ on the closed
invariant subspace $Y_f:=\ol{\lin}\chi\S f$ for some given $f\in
C(K)$. This extends results of Robinson \cite[Theorem~1.1]{robinson94} and
Lenz \cite[Theorem~1]{lenz09}. 

For a closed subspace $H\subset L^2(K,\mu)$ we denote by $P_H$ the
orthogonal projection onto $H$.

\begin{thm}
  \label{thm:mean-ergodic-in-f}
Let $\S=\{S_g:g\in G\}$ be a representation of a (right) amenable
semigroup $G$ as Markov operators on $C(K)$ and assume that $\S$ is
uniquely ergodic with invariant measure $\mu$.
For $\chi\in\widehat{G}$ and $f\in C(K)$ the following assertions are
equivalent.
\begin{enumerate}[(1)]
\item $P_{\Fix\chi\S_2}f\in \Fix\chi\S$.
\item $\chi\S$ is mean ergodic on $Y_f$ with mean ergodic projection $P_\chi$.
\item $\Fix\chi\S|_{Y_f}$ separates $\Fix(\chi\S)|_{Y_f}'$.
\item $f\in \Fix\chi\S\oplus\ol{\lin}\rg(I-\chi\S)$.
\item $A_\a^{\chi\S} f$ converges weakly (to a fixed point of
  $\chi\S$) for some/every strong (right) $\chi\S$-ergodic
  net $(A_\a^{\chi\S})$.
\item $A_\a^{\chi\S} f$ converges strongly (to a fixed point of
  $\chi\S$) for some/every strong (right)
  $\chi\S$-ergodic net $(A_\a^{\chi\S})$.
\end{enumerate}
The limit $P_\chi$ of $A_\a^{\chi\S}$ in the strong (weak, resp.) operator
topology on $Y_f$ is the mean ergodic projection of
$\chi\S|_{Y_f}$ mapping $Y_f$ onto $\Fix\chi\S|_{Y_f}$ along
$\ol{\lin}\rg(I-\chi\S|_{Y_f})$.
\end{thm} 

\begin{proof}
  The equivalence of the statements (2) to (6) follows directly from
Proposition~1.11 in \cite{schreiber12}.

(6)$\imp$(1): 
By von Neumann's Ergodic Theorem $P_{\Fix\chi\S_2}f$
is the limit of $A_\a^{\chi\S} f$ in $L^2(K,\mu)$. 
If $A_\a^{\chi\S} f$ converges strongly in $C(K)$, then the limits
coincide almost everywhere and hence $P_{\Fix\chi\S_2}f$ has a
continuous representative in $\Fix\chi\S$.

(1)$\imp$(4): 
Let $\nu\in C(K)'$ vanish on
$\Fix\chi\S\oplus\ol{\lin}\rg(I-\chi\S)$. Then, in particular,
$\sk{\nu,h}=\sk{\nu,\chi(g)S_g h}=\sk{(\chi(g)S_g)' \nu, h}$ for all
$h\in C(K)$ and $g\in G$ and thus $\nu\in \Fix(\chi\S)'$. Hence by Lemma
\ref{lemma:Fixraum_dualer_Fixraum} there exists $h\in
\Fix(\chi\S_2)^*$ such that $\nu=\ol{h} d\mu$. Let $(A_\a^{\chi\S_2})_{\a\in\mathcal{A}}$
be a strong $\chi\S_2$-ergodic net on $L^2(K,\mu)$. Then
$(A_\a^{\chi\S_2})^* h=h$ for all $\a\in\mathcal{A}$ and von Neumann's
Ergodic Theorem implies
$$\sk{\nu,f}=\sk{f,h}_{L^2}=\sk{A_\a^{\chi\S_2}
  f,h}_{L^2}\to \langle  \underbrace{P_{\Fix\chi\S_2}f}_{\in
    C(K)},h\rangle_{L^2}= \langle\nu,\underbrace{P_{\Fix\chi\S_2}f}_{\in
    \Fix\chi\S}\rangle=0.$$
Hence the Hahn-Banach Theorem yields $f\in
\Fix\chi\S\oplus\ol{\lin}\rg(I-\chi\S)$ since
$\Fix\chi\S\oplus\ol{\lin}\rg(I-\chi\S)$ is closed by Theorem~1.9 in
Krengel \cite[Chap.~2]{krengel85}.
\end{proof}

\begin{remark}
 In \cite[Theorem~1]{lenz09} Lenz showed that for Koopman respresentations of locally compact, $\sigma$-compact abelian groups the assertion (1) of
 Theorem~\ref{thm:mean-ergodic-in-f} is equivalent to the convergence
  of $\chi\S$-ergodic nets associated to so-called \emph{van
    Hove sequences} (see
  Schlottmann \cite[p.~145]{schlottmann00}), a special class of
  \emph{F\o lner sequences} (see Paterson \cite[Chapter~4]{paterson88}).
 \end{remark}

 \begin{remark}
\label{rem:p_chi}
   Notice that if $P_{\Fix\chi\S_2}f=0$ in the situation of Theorem~\ref{thm:mean-ergodic-in-f} then $\chi\S$ is mean ergodic on $Y_f$
with mean ergodic projection $P_\chi=0$.
To see this, let $\nu\in C(K)'$ vanish on
$\ol{\lin}\rg(I-\chi\S)$.
Then the same argument as in the proof of the implication (1)$\imp$(4) in Theorem~\ref{thm:mean-ergodic-in-f} shows that $\nu$ vanishes in $f$, which
yields the claim.
 \end{remark}

We now recall the concept of a uniform family of ergodic nets from \cite{schreiber12} and apply it to operators on the Banach space $C(K)$.

\begin{definition}
\label{definition:uniformly-family}
Suppose that the semigroup $G$
is represented on $C(K)$ by bounded semigroups
$\S_i=\{S_{i,g}:g\in G\}$ for each $i$ in some index set $I$ such that
the $\S_i$ are \emph{uniformly bounded}, i.e.,  
$\sup_{i\in I}\sup_{g\in G}\|S_{i,g}\|<\infty$. 
Let $\A$ be a directed set and let $(A_\a^{\S_i})_{\a\in\A}\subset
\L(C(K))$ be a net of operators for each $i\in I$. 
Then $\{(A_\a^{\S_i})_{\a\in\A}: i\in I\}$ is a \emph{uniform family
  of right (left) ergodic nets} if
\begin{enumerate}
\item  $\forall\a\in\A, \forall \e>0, \forall f_1,\dots, f_m\in C(K), \exists
  g_1,\dots, g_n\in G$ such that for each $i\in I$ there exists a convex
  combination $\sum_{j=1}^{n}c_{i,j} S_{i,g_j}\in \co\S_i$
  satisfying
$$\sup_{i\in I}\|A_\a^{\S_i}f_k-\textstyle{\sum_{j=1}^{n}}c_{i,j}
S_{i,g_j}f_k\|_\infty<\e\quad \forall k\in\{1,\dots, m\};$$
\item $\displaystyle \lim_\a \sup_{i\in
    I}\|A_\a^{\S_i}f-A_\a^{\S_i}S_{i,g}f\|_\infty=0\en \left( \lim_\a \sup_{i\in
    I}\|A_\a^{\S_i}f-S_{i,g}A_\a^{\S_i}f\|_\infty= 0\right) \en
\forall g\in G, f\in~C(K).$
\end{enumerate}
The set $\{(A_\a^{\S_i})_{\a\in\A}: i\in I\}$ is called a \emph{uniform family
  of ergodic nets} if it is a uniform family of left and right ergodic nets.
\end{definition}
Notice that if
$\{(A_\a^{\S_i})_{\a\in\A}: i\in I\}$ is a uniform 
family of (right) ergodic nets, then each $(A_\a^{\S_i})_{\a\in\A}$ is a
strong (right) $\S_i$-ergodic net.
The simplest non-trivial example of a uniform family of ergodic nets
is the family of weighted Cesàro means
$$\left\{\left(\frac{1}{n}\sum_{j=1}^{n}(\lambda
    S)^j\right)_{n\in\N}: \lambda\in\Torus\right\}$$ 
for a contraction $S\in \L(C(K))$.
See \cite[Proposition 2.2]{schreiber12} for more examples.

We now choose a subset $\Lambda$ of characters in $\widehat{G}$ and
consider the semigroups $\chi\S$ 
for each $\chi\in\Lambda$. 
If $\{(A_\a^{\chi\S})_{\a\in\mathcal{A}}: \chi\in\Lambda\}$ is a uniform family of right
ergodic nets, $f\in C(K)$ and $\chi\S$ is right amenable and mean
ergodic on $\ol{\lin}\chi\S f$ with mean ergodic projection $P_\chi$
for each $\chi\in \Lambda$, then $A_\a^{\chi\S} f$ converges (in the 
supremum norm) to $P_\chi f$ for each 
$\chi\in\Lambda$ by Theorem~\ref{thm:mean-ergodic-in-f}. 
The next corollary gives a sufficient condition for
this convergence to be uniform in $\chi\in\Lambda$. 
It generalizes Theorem~2 of Lenz~\cite{lenz09} to right amenable
semigroups of Markov operators.

\begin{cor}
  \label{cor:uniform-convergence_C_K}
Let $G$ be represented on $C(K)$ by a  right amenable
semigroup $\S=\{S_g:g\in G\}$ of Markov operators and let $\S$ be
uniquely ergodic with invariant measure
$\mu$. Consider the semigroups $\chi\S$ for each $\chi$ in a compact
set $\Lambda\subset\widehat{G}$. If $\{(A_\a^{\chi\S})_{\a\in \A}:
\chi\in \Lambda\}$ is a uniform family of right ergodic nets and if $f\in C(K)$ satisfies
\begin{enumerate}
\item $P_{\Fix\chi\S_2}f\in C(K)$ for each $\chi\in \Lambda$,
\item the map $\Lambda \to \R_+, \chi\mapsto\|A_\a^{\chi\S}f-
  P_\chi f\|_\infty$ is continuous for each $\a\in\A$,
\end{enumerate}
then 
$$\lim_\a\sup_{\chi\in \Lambda}\|A_\a^{\chi\S}f-P_\chi f\|_\infty = 0.$$ 
\end{cor}

\begin{proof}
 By Theorem \ref{thm:mean-ergodic-in-f} and our
 hypotheses, the semigroup $\chi\S$ is mean ergodic on $\ol{\lin}\chi\S
 f$ for all $\chi\in\Lambda$. The result then follows directly from
 Theorem~2.4 in \cite{schreiber12}. 
\end{proof}

The following corollary is a direct consequence. It generalizes
Theorem~2.10 of Assani~\cite{assani03}, who considered Koopman
representations of the semigroup
$(\N,+)$ and the F\o lner sequence given by $F_n=\{0,1,\dots,n-1\}$.
\begin{cor}
\label{cor:assani}
  Let $H$ be a locally compact group with left Haar measure $|\cdot|$ and
  suppose that $G\subset H$ is a subsemigroup such that there exists a F\o
  lner net $(F_\a)_{\a\in\A}$ in $G$. Let $G$ be represented on $C(K)$
  by  a  semigroup $\S=\{S_g: g\in G\}$ of Markov operators and
  assume that $\S$ is uniquely ergodic with invariant
  measure $\mu$. If $f\in C(K)$ satisfies $P_{\Fix\chi\S_2}f=0$ for
  all $\chi$ in a compact set $\Lambda\subset\widehat{G}$, then 
$$\lim_{\a}\sup_{\chi\in\Lambda}\left\|\frac{1}{|F_\a|}\int_{F_\a}\chi(g)S_g
  f \;dg\right\|_\infty=0.$$
\end{cor}

\begin{proof}
  If $(F_\a)$ is a F\o lner net in $G$, then $G$ and consequently $\S$
  is right amenable. 
Since $\Lambda\subset \widehat{G}$ is compact, it
  follows from \cite[Proposition 2.2 (f)]{schreiber12} that
$$\left\{\left(\frac{1}{|F_\a|}\int_{F_\a}\chi(g)S_g
  \;dg\right)_{\a\in\A}: \chi\in\Lambda\right\}$$
is a uniform family of right ergodic nets. 
If $P_{\Fix\chi\S_2}f=0$ for all
$\chi\in\Lambda$, then by Remark~\ref{rem:p_chi} the conditions (1) and (2) of Corollary
\ref{cor:uniform-convergence_C_K} are satisfied since the map
$\chi\mapsto \frac{1}{|F_\a|}\int_{F_\a}\chi(g)S_g
  f dg$ is continuous. Hence 
$$\lim_{\a}\sup_{\chi\in\Lambda}\left\|\frac{1}{|F_\a|}\int_{F_\a}\chi(g)S_g
  f \;dg\right\|_\infty=0.$$
\end{proof}

\section{Semigroups of Koopman operators}
\label{sec:koopman}

In this section we consider semigroups of Koopman
operators on the space $C(K,\C^N)$ of continuous $\C^N$-valued
functions on a compact space $K$ for some $N\in\N$. 
The space $\C^N$ will be endowed with the Euclidean norm $x\mapsto\|x\|_2=\sqrt{\sk{x,x}_2}$
and the space $C(K,\C^N)$ with the norm
$f\mapsto\|f\|=\sup_{x\in K}\|f(x)\|_2$.
We identify $C(K,\C^N)$ with $C(K)^N$ and write
$f=(f_1,\dots,f_N)\in C(K,\C^N)$ with coordinate functions $f_i\in
C(K)$.
As before, $G$ is a semitopological semigroup.

\begin{definition}
 A \emph{semigroup action of $G$ on $K$} is a continuous map
  $$G\times K\to K,\en (g,x)\mapsto gx$$ 
satisfying
$$(g_1 g_2)x=g_1(g_2x)\text{ for all } g_1,g_2\in G \text{ and } x\in K.$$
In this case we say that \emph{$G$ acts on $K$}.
\end{definition}

Let $G$ be a semitopological semigroup acting on $K$ and let $\S:=\{S_g: g\in G\}$ be the corresponding 
\emph{Koopman representation} on
$C(K,\C^N)$, i.e., $S_g f(x)=f(gx)$ for  $f\in C(K,\C^N)$, $g\in G$ and $x\in
K$. 
To emphasize the dependence on $N$ we sometimes write $\S^{(N)}$ for
the semigroup $\S$ on $C(K,\C^N)$.



We say that a measure $\mu$ on $K$ is \emph{$G$-invariant} if
$\mu(A)=\mu(g^{-1}A)$ for all Borel sets $A\subseteq K$, where
$g^{-1}A=\{x\in K: gx\in A\}$. Notice that this is equivalent to
$\mu\in{\Fix\S^{(1)}}'$. 
If $\mu$ is a $G$-invariant measure, we denote by
$\S_2:=\S_2^{(N)}:=\{S_{g,2}: g\in G\}$ the extension of the semigroup 
$\S^{(N)}$ to $L^2(K,\C^N,\mu)$.
The action of $G$ on $K$ is called \emph{ergodic with respect to $\mu$} if there is no
non-trivial measurable $G$-invariant set, or equivalently if $\Fix\S_2^{(1)}=\C\cdot
\1\subseteq L^2(K,\mu)$.
The action of $G$ on $K$ is called \emph{uniquely ergodic} if there
exists a unique $G$-invariant probability measure $\mu$ on
$K$. Notice that this is equivalent to $\Fix{\S^{(1)}}'=\C\cdot\mu$ for some probability measure $\mu\in C(K)'$.

\begin{definition}
Let $\Omega$ be a topological group. A continuous map $\gamma:G\times
  K\to\Omega$ is called a \emph{continuous cocycle} if it satisfies the
\emph{cocycle equation}
$$\gamma(g_1g_2,x)=\gamma(g_2,x)\gamma(g_1,g_2x)\quad\forall
g_1,g_2\in G, x\in K.$$
The set of continuous cocycles is denoted by $\Gamma(G\times K,
\Omega)$. 
If $\Omega$ is a compact metric group with metric $d$, then we endow
$\Gamma(G\times K,\Omega)$ with the metric
$$\tilde{d}(\gamma_1,\gamma_2):=\sup_{(g,x)\in G\times
  K}d\left(\gamma_1(g,x),\gamma_2(g,x)\right)\quad (\gamma_1,\gamma_2\in \Gamma(G\times K,\Omega)).$$
\end{definition}
Denote by $U(N)$ the group of unitary operators on $\C^N$ and take
 a continuous cocycle $\gamma\in\Gamma(G\times K,U(N))$. 
Motivated by papers of Walters~\cite{walters96} and Santos and
Walkden~\cite{santos07}
we study the mean ergodicity of the semigroup
$\gamma\S:=\{\gamma(g,\cdot)S_g: g\in G\}$ on $C(K,\C^N)$, where
$\gamma(g,\cdot)S_g f(x)=\gamma(g,x)S_gf(x)$ for $f\in C(K,\C^N)$ and
$x\in K$. 

In order to proceed as in the previous section we need some facts
about vector valued measures (see Diestel and Uhl~\cite{diestel77}). 
Denote by $M(K,\C^N)$ the set of $\sigma$-additive
functions $\nu: \Sigma\to \C^N$ defined on the Borel $\sigma$-algebra
$\Sigma$ of $K$. 
We define the \emph{total variation} $|\nu|_2:\Sigma\to
[0,\infty]$ of a measure $\nu\in M(K,\C^N)$ by 
$$|\nu|_2(E):=\sup\left\{\sum_{j=1}^\infty \|\nu(E_j)\|_2 :
   E=\bigsqcup_{j\in\N}E_j\right\},\quad (E\in\Sigma),$$
where $E=\bigsqcup_{j\in\N} E_j$ means that the family $(E_j)_{j\in\N}\subset \Sigma$ is a partition of $E$.

The main property of the total variation of a measure $\nu\in M(K,\C^N)$ is the fact that $|\nu|_2:\Sigma\to\R_+$ is a finite positive measure on $K$, which can be deduced from Theorem~6.2 and Theorem~6.4 in Rudin~\cite{rudin87}.



Identifying $M(K,\C^N)$ with $M(K)^N$, we take $\nu=(\nu_1,\dots,\nu_N)\in M(K,\C^N)$. If $f=(f_1,\dots,f_N)\in C(K,\C^N)$, then the
map
$$\tilde{d}\nu: f \mapsto \sum_{i=1}^N \int_K f_i d\nu_i$$
defines a linear functional on $C(K,\C^N)$. For $f\in
C(K,\C^N)$ we have
\begin{align*}
\left | \sk{\tilde{d}\nu,f}\right|&=\left|\sum_{i=1}^N \int_K f_i d\nu_i\right|\le \sum_{i=1}^N \int_K |f_i| d|\nu_i|\\
&\le \sum_{i=1}^N \sup_{x\in K}|f_i(x)||\nu_i|(K)\le
\max_{i=1,\dots,N}\sup_{x\in K}|f_i(x)|\sum_{i=1}^N |\nu_i|(K)\\
&\le\sup_{x\in K}\left(\sum_{i=1}^N|f_i(x)|^2\right)^{\frac{1}{2}}\sum_{i=1}^N
|\nu_i|(K)=\|f\| \sum_{i=1}^N |\nu_i|(K)
\end{align*}
and hence $\tilde{d}\nu$ is bounded with $\|\tilde{d}\nu\|\le \sum_{i=1}^N |\nu_i|(K)$.

The next result follows from the Riesz Representation Theorem
 and is in fact equivalent to it (cf. \cite[Chapter~VI, Section~7, Theorem~3]{dunford-schwartz58}).
\begin{thm}
\label{thm:riesz}
  The map 
  \begin{align*}
    \tilde{d}: M(K,\C^N)&\to C(K,\C^N)'\\
\nu&\mapsto \tilde{d}\nu
  \end{align*}
is linear and bijective.
\end{thm}
\begin{proof}
 The only non-trivial statement is the surjectivity. So take $\xi\in
 C(K,\C^N)'$, $\{e_1,\dots,e_N\}$ the canonical basis of $\C^N$
 and define $\xi_i\in C(K)'$ for each $i\in \{1,\dots,N\}$ by $\xi_i(f):=\xi(f\otimes
 e_i)$, where $f\otimes e_i\in C(K,\C^N)$ is the function $x\mapsto f(x)e_i$.
By the Riesz Representation Theorem for each $i\in\{1,\dots,N\}$ there
 exists $\nu_i\in M(K)$ with $\xi_i=d\nu_i$. If we define
 $\nu:=(\nu_1,\dots,\nu_N)\in M(K,\C^N)$, then for each $f=(f_1,\dots,f_N)\in C(K,\C^N)$ we obtain
 $$\sk{\tilde{d}\nu,f}= \sum_{i=1}^N \int_K f_i
 d\nu_i=\sum_{i=1}^N \sk{\xi_i,f_i} =\sk{\xi,\sum_{i=1}^N f_i\otimes
   e_i}=\sk{\xi,f}$$
and hence $\tilde{d}\nu=\xi$.
\end{proof}

For a bounded linear functional $\nu\in C(K,\C^N)'$ we define the
functional $|\nu|_2$ by
$$\sk{|\nu|_2,f}:=\sup\left\{|\sk{\nu,h}|: h\in C(K,\C^N),
  \|h(\cdot)\|_2\le f\right\}$$
for $0\le f\in C(K)$. 
It is clear that $\|\nu\|=\sk{|\nu|_2,\1}$.

\begin{prop}
  There exists a unique bounded and linear extension of $|\nu|_2$ to $C(K)$ with $\||\nu|_2\|=\|\nu\|$.
\end{prop}
\begin{proof}
The positive homogeneity of $|\nu|_2$ is clear from the definition. To
see additivity, take $0\le f_1,f_2\in C(K)$ and $\|h_1(\cdot)\|_2\le
f_1$ and $\|h_2(\cdot)\|_2\le f_2$. Then we have for certain
$c_1,c_2\in\Torus$
\begin{align*}
  |\sk{\nu,h_1}|+|\sk{\nu,h_2}|&=|c_1\sk{\nu,h_1}+c_2\sk{\nu,h_2}| \le
  \sk{|\nu|_2,\|c_1 h_1(\cdot)+c_2 h_2(\cdot)\|_2} \\
&\le
  \sk{|\nu|_2,\|h_1(\cdot)\|_2+\|h_2(\cdot)\|_2} \le \sk{|\nu|_2,f_1+f_2}
\end{align*}
and thus $\sk{|\nu|_2,f_1}+\sk{|\nu|_2,f_2}\le \sk{|\nu|_2,f_1+f_2}$. 

For the converse inequality take $\|h(\cdot)\|_2\le f_1+f_2$ and
$\e>0$. The open sets 
$$U_1:=\{x\in K: \|h(x)\|_2>0\}\en \text{ and }\en
U_2:=\{x\in K: \|h(x)\|_2<\e\}$$
cover $K$. Hence by Theorem D.6 in \cite{efhn} there exists a
function $\psi\in C(K)$ with $0\le\psi\le\1$ and 
$\supp(\psi)\subset U_1$ and $\supp(\1-\psi)\subset U_2$. 
Define for $j=1,2$
$$h_j(x):=\left\{
  \begin{array}{ll}
    \psi(x)\frac{f_j(x)}{f_1(x)+f_2(x)}h(x),& f_1(x)+f_2(x)\neq 0\\
0,& \text{else}.
  \end{array}
\right. $$
Then we have $h_j\in C(K)$, $h_1+h_2=\psi h$ and $\|h_j(\cdot)\|_2\le
f_j$ for each $j=1,2$. Moreover, we obtain 
$$\|h(\cdot)-(h_1+h_2)(\cdot)\|_2=\|(\1-\psi)h(\cdot)\|_2=|\1-\psi|\|h(\cdot)\|_2<
\e \1$$
and thus 
\begin{align*}
  |\sk{\nu,h}|&\le |\sk{\nu,h_1+h_2}|+\e\|\nu\|\le
  \sk{|\nu|_2,f_1}+\sk{|\nu|_2,f_2}+\e\|\nu\|. 
\end{align*}
Hence $\sk{|\nu|_2,f_1+f_2}\le
\sk{|\nu|_2,f_1}+\sk{|\nu|_2,f_2}+\e\|\nu\|$ and thus 
 $$\sk{|\nu|_2,f_1+f_2}\le \sk{|\nu|_2,f_1}+\sk{|\nu|_2,f_2}$$ 
by letting $\e\downarrow 0$.

Finally, we extend $|\nu|_2$ first to $C(K,\R)$ by 
$$\sk{|\nu|_2,f}:= \sk{|\nu|_2,f^+}-\sk{|\nu|_2,f^-}\quad (f\in C(K,\R)),$$
where $f^+:=\sup\{f,0\}$ and $f^-:=\sup\{-f,0\}$,
and then to $C(K)$ by
$$\sk{|\nu|_2,f}:=\sk{|\nu|_2,\Re f}+i \sk{|\nu|_2,\Im f}\quad (f\in
C(K)).$$
It is straightforward to check that in this way $|\nu|_2$ becomes
linear on $C(K)$.
The boundedness of $|\nu|_2$ follows from
$$\sk{|\nu|_2,f}\le \sk{|\nu|_2,\|f\|_\infty \1}=\|f\|_\infty \|\nu\| \quad (0\le f\in C(K)).$$
This implies $\||\nu|_2\|\le\|\nu\|$, and equality follows from $\sk{|\nu|_2,\1}=\|\nu\|$.
\end{proof}
The notation $|\nu|_2$ for a
functional $\nu\in C(K,\C^N)'$ and a measure $\nu\in M(K,\C^N)$ is
justified by the following theorem.

\begin{thm}
\label{thm:commutative-diagram}
  The following diagram commutes.
    $$\begin{CD}
      M(K,\C^N) @>\tilde{d}>>C(K,\C^N)' \\
      @VV|\cdot|_2V @VV|\cdot|_2V\\
      M(K) @>d>> C(K)'
    \end{CD}$$
\end{thm}
\begin{proof}
Let $\nu=(\nu_1,\dots,\nu_N)\in M(K,\C^N)$ and $0\le f\in C(K)$. 
We consider $C(K,\C^N)$ as a dense subspace of $L^1(K,\nu)$ and obtain
\begin{align*}
  \sk{|\tilde{d}\nu|_2,f}&=\sup\left\{\left|\sk{\tilde{d}\nu,h}\right|:
    h\in C(K,\C^N), \|h(\cdot)\|_2\le f\right\} \\
& =\sup \left\{\left|\sk{\tilde{d}\nu,h}\right|:  0\le\sum_{j\in\N}\beta_j\1_{E_j}\le f, \bigsqcup_j E_j=K, \|h(\cdot)\|_2\le\sum_{j\in\N}\beta_j\1_{E_j}\right\}\\
&=\sup \left\{\left|\sum_{i=1}^N\sum_{j,l\in\N}\alpha_{j,l}^{(i)}\nu_i(E_{j,l})\right|   : {0\le\sum_{j\in\N}\beta_j\1_{E_j}\le f, \bigsqcup_j E_j=K, \atop \sum_{j,l\in\N}\|\alpha_{j,l}\|_2\1_{E_{j,l}}\le\sum_{j\in\N}\beta_j\1_{E_j}, \bigsqcup_l E_{j,l}=E_j}\right\}\\
&=\sup \left\{\left|\sum_{j,l\in\N}\sk{\alpha_{j,l},\nu(E_{j,l})}_2\right|   : {0\le\sum_{j\in\N}\beta_j\1_{E_j}\le f, \bigsqcup_j E_j=K, \atop \sum_{j,l\in\N}\|\alpha_{j,l}\|_2\1_{E_{j,l}}\le\sum_{j\in\N}\beta_j\1_{E_j}, \bigsqcup_l E_{j,l}=E_j}\right\}
\end{align*}
Under the condition $\|\alpha_{j,l}\|_2\le\beta_j$ for all $j,l\in\N$, the expression $\left|\sum_{j,l\in\N}\sk{\alpha_{j,l},\nu(E_{j,l})}_2\right|$ becomes maximal if $\alpha_{j,l}=\beta_j\frac{\nu(E_{j,l})}{\|\nu(E_{j,l})\|_2}$ for all $j,l\in\N$. 
Hence 
\begin{align*}
   \sk{|\tilde{d}\nu|_2,f}
&=\sup \left\{\sum_{j,l\in\N}\beta_j\|\nu(E_{j,l})\|_2 : {\sum_{j\in\N}\beta_j\1_{E_j}\le f, \bigsqcup_j E_j=K,\,  \bigsqcup_l E_{j,l}=E_j}\right\}\\
&=\sup \left\{\sum_{j\in\N}\beta_j|\nu|_2(E_{j}) : {\sum_{j\in\N}\beta_j\1_{E_j}\le f, \,\bigsqcup_j E_j=K}\right\}\\
&=\sup \left\{\int_K\sum_{j\in\N}\beta_j\1_{E_j}d|\nu|_2 : {\sum_{j\in\N}\beta_j\1_{E_j}\le f, \,\bigsqcup_j E_j=K}\right\}\\
&=\int_K f\, d|\nu|_2=\sk{d|\nu|_2,f}
\end{align*}
and thus $|\tilde{d}\nu|_2=d|\nu|_2$.
\end{proof}

By virtue of Theorem \ref{thm:riesz} and Theorem
\ref{thm:commutative-diagram} we shall identify $M(K,\C^N)$ with
$C(K,\C^N)'$ and we will use the same notation $|\nu|_2$ for a measure $\nu\in M(K,\C^N)$ and
a functional $\nu\in C(K,\C^N)'$ without explicitly distinguishing
these two objects. 

We now return to the situation of the beginning of this section and
characterize the mean ergodicity of $\gamma\S$ for a continuous
cocycle $\gamma\in \Gamma(G\times K, U(N))$.
For a function $h\in L^2(K,\C^N,\mu)$ we denote by $\ol{h} d\mu\in
C(K,\C^N)'$ the functional defined by 
$$\sk{\ol{h} d\mu,f}:=\sk{f,h}_{L^2(K,\C^N,\mu)}=\sum_{i=1}^N\int_K
f_i(x)\ol{h_i(x)}d\mu(x)\quad (f\in C(K,\C^N)).$$

\begin{lemma}
  \label{lemma:Fixraum_dualer_Fixraum_cocycles}
Let the action of $G$ on $K$ be uniquely ergodic with invariant
measure $\mu$ and let $\S$ and $\S_2$ be the corresponding Koopman
representations on $C(K,\C^N)$ and $L^2(K,\C^N,\mu)$, respectively.
If $\gamma:G\times K\to U(N)$ is a continuous cocycle, then the map 
\begin{align*}
 L^2(K,\C^N,\mu)\supseteq\Fix( \gamma\S_2)^*&\to\Fix  (\gamma\S)'\subseteq C(K,\C^N)'\\
h\quad&\mapsto\quad \ol{h} d\mu
\end{align*}
is antilinear and bijective.
\end{lemma}
\begin{proof}
    To see that the map is well defined, take $h\in\Fix(\gamma\S_2)^*$. 
  Then for all $f\in C(K,\C^N)$ and $g\in G$  we have
 \begin{align*}
  \sk{(\gamma(g,\cdot)S_g)'(\ol{h} d\mu),f}&=\sk{\ol{h} d\mu,\gamma(g,\cdot)S_g f}
=\sk{\gamma(g,\cdot)S_{g,2} f,h}_{L^2(K,\C^N,\mu)}\\
&=\sk{f,(\gamma(g,\cdot)S_{g,2})^*h}_{L^2(K,\C^N,\mu)}\\
&=\sk{f,h}_{L^2(K,\C^N,\mu)}
=\sk{\ol{h} d\mu,f},
  \end{align*}
 yielding $\ol{h} d\mu\in\Fix (\gamma\S)'$.

%
As antilinearity and injectivity are clear, it remains to show
surjectivity. Let $\nu=(\nu_1,\dots,\nu_N)\in\Fix (\gamma\S)'$. We
claim that $|\nu|_2\le
S_g'|\nu|_2$ for all $g\in G$. Indeed, if $0\le f\in C(K)$ and $g\in G$,
then
\begin{align*}
  \sk{|\nu|_2,f}&=\sup_{\|h(\cdot)\|_2\le f}|\sk{\nu,h}|\\
&=\sup_{\|h(\cdot)\|_2\le
    f}|\sk{\nu,\gamma(g,\cdot)S_g h}|\\
&\le \sup_{\|h(\cdot)\|_2\le f}\sk{|\nu|_2,\|\gamma(g,\cdot) S_g h(\cdot)\|_2}\\
&= \sup_{\|h(\cdot)\|_2\le f}\sk{|\nu|_2,S_g\|h(\cdot)\|_2}\\
&=\sk{|\nu|_2,S_g f}=\sk{S_g'|\nu|_2,f},
\end{align*}
since $\gamma(g,x)$ is unitary for all $x\in K$.

 If $0\le f\in C(K)$ and $g\in G$, then 
$$0\le\sk{S_g'|\nu|_2-|\nu|_2,f}\le
\sk{S_g'|\nu|_2-|\nu|_2,\|f\|_\infty\1}=\|f\|_\infty\sk{|\nu|_2,S_g\1-\1}=0$$ 
and thus $|\nu|_2\in \Fix\S'=\C\cdot \mu$ by unique ergodicity. 
As a consequence of Theorem~\ref{thm:commutative-diagram} and since
$|\nu_i|\le |\nu|_2$ the measures $\nu_i$ 
are thus absolutely continuous with respect to $\mu$ for each
$i=1,\dots,N$. 
The Radon-Nikodým Theorem then implies
  the existence of functions $h_i\in L^\infty(K,\mu)$  such that $\nu_i=\ol{h_i} 
  d\mu$ for all $i=1,\dots,N$. 
Defining $h:=(h_1,\dots,h_N)\in
  L^\infty(K,\C^N,\mu)$ we obtain $\nu=\ol{h}  d\mu$ and the same
  calculation as above shows that $h\in \Fix(\gamma \S_2)^*$.
%
\end{proof}

\begin{lemma}
\label{lemma:dimension-fixraum2}
  Let the action of a right amenable semigroup $G$ on $K$ be ergodic
  with respect to some invariant measure $\mu$ and let $\S_2$ be the
  corresponding Koopman representation on $L^2(K,\C^N,\mu)$. If $\gamma: G\times
  K\to U(N)$ is a continuous cocycle, then $\dim \Fix\gamma\S_2\le N$.
\end{lemma}
\begin{proof}
Suppose $\dim \Fix\gamma\S_2> N$ and take $N+1$ linearly independent
functions $f_1,\dots, f_N,h\in \Fix\gamma\S_2$. 
We may assume that $\|f_i(\cdot)\|_2=\1$ for each $i\in\{1,\dots,N\}$
since if $f_i\in \Fix\gamma\S_2$ then $\|f_i(\cdot)\|_2\in\Fix\S_2$
and thus $\|f_i(\cdot)\|_2$ is constant by ergodicity. 
Moreover, by a pointwise application of the Gram-Schmidt process, we
may assume that $\sk{f_i(x),f_j(x)}_2=\delta_{ij}$ for $\mu$-a.e. $x\in
K$ and each $i,j\in\{1,\dots,N\}$.
Hence $h$ can be written as
$$h(x)=\sum_{i=1}^N \sk{h(x),f_i(x)}_2f_i(x)\qquad \mu\text{-a.e. }x\in K.$$

For each $i\in\{1,\dots, N\}$ we define the function $h\bullet f_i$ by $h\bullet
f_i(x):=\sk{h(x),f_i(x)}_2$ for $\mu$-a.e. $x\in K$ and claim that $h\bullet f_i$ is constant. 
Indeed, for each $i\in\{1,\dots,N\}$, $g\in G$ and $\mu$-a.e. $x\in K$ we have
\begin{align*}
  S_{g,2}(h\bullet
  f_i)(x)&=\sk{h(gx),f_i(gx)}_2=\sk{\gamma(g,x)^{-1}h(x),\gamma(g,x)^{-1}f_i(x)}_2\\
&=\sk{h(x),f_i(x)}_2=h\bullet f_i(x)
\end{align*}
since $\gamma(g,x)$ is unitary. Hence $h\bullet f_i\in\Fix\S_2^{(1)}$ and
thus $h\bullet f_i\in\C\cdot\1$ by ergodicity. 
Hence $h$ is a linear combination of $f_1,\dots, f_N$ contradicting
the linear independence.
\end{proof}

The following theorem is the analogue of
Theorem~\ref{thm:mean-ergodic} for cocycles and generalizes Theorem~4
of Walters \cite{walters96} to amenable semigroups. 

\begin{thm}
  \label{thm:mean-ergodic-koopman}
Let the action of a (right) amenable semigroup $G$ on $K$ be uniquely ergodic
with invariant measure $\mu$ and let $\S$ and $\S_2$ be the
  corresponding Koopman representations on $C(K,\C^N)$ and
  $L^2(K,\C^N,\mu)$, respectively. 
If $\gamma: G\times K\to U(N)$ is a
continuous cocycle, then the following assertions are equivalent.
\begin{enumerate}[(1)]
\item $\Fix\gamma\S_2\subseteq \Fix\gamma\S$.
\item $\gamma\S$ is mean ergodic on $C(K,\C^N)$ with mean ergodic projection $P_\gamma$.
\item $\Fix\gamma\S$ separates $\Fix(\gamma\S)'$.
\item $C(K,\C^N)= \Fix\gamma\S\oplus\ol{\lin}\rg(I-\gamma\S)$.
\item $A_\a^{\gamma\S} f$ converges weakly (to a fixed point of
  $\gamma\S$) for some/every strong (right) $\gamma\S$-ergodic
  net $(A_\a^{\gamma\S})$ and all $f\in C(K,\C^N)$.
\item $A_\a^{\gamma\S} f$ converges strongly (to a fixed point of
  $\gamma\S$) for some/every strong (right) $\gamma\S$-ergodic net
  $(A_\a^{\gamma\S})$ and all $f\in C(K,\C^N)$.
\end{enumerate}
The limit $P_\gamma$ of the nets $(A_\a^{\gamma\S})$ in the strong (weak, resp.)
operator topology is the mean ergodic projection of
$\gamma\S$ mapping $C(K,\C^N)$ onto $\Fix \gamma\S$ along
$\ol{\lin}\rg(I-\gamma\S)$.
\end{thm} 
\begin{proof}
  The equivalence of the statements (2) to (6) follows directly from
Theorem~1.7 and
Corollary~1.8 in \cite{schreiber12}. 

Notice that $\Fix(\gamma\S_2)^*=\Fix\gamma\S_2$ since $\gamma\S_2$
consists of contractions on $L^2(K,\C^N,\mu)$.

(1)$\imp$(3): If $0\neq\nu\in \Fix(\gamma\S)'$, then $\nu=\ol{h}d\mu$ by
Lemma~\ref{lemma:Fixraum_dualer_Fixraum_cocycles} for some $0\neq h \in
\Fix(\gamma\S_2)^*=\Fix\gamma\S_2$. Since $\Fix\gamma\S_2\subseteq \Fix\gamma\S$
by (1), this yields $h\in \Fix\gamma\S$ and
$$\sk{\nu,h}=\|h\|^2_{L^2(K,\C^N,\mu)}>0.$$

(3)$\imp$(1): Suppose $f\in \Fix\gamma\S_2\setminus \Fix\gamma\S$. 
By Lemma~\ref{lemma:dimension-fixraum2} the space $\Fix\gamma\S_2$ is finite
dimensional and by Lemma~\ref{lemma:Fixraum_dualer_Fixraum_cocycles}
we thus have
$$\dim\Fix\gamma\S< \dim\Fix\gamma\S_2=\dim\Fix(\gamma\S)'.$$ 
Hence $\Fix\gamma\S$ does not separate $\Fix(\gamma\S)'$. 
\end{proof}

The following theorem characterizes mean ergodicity of $\gamma\S$
on the closed invariant subspace $Y_f:=\ol{\lin}\gamma\S f$ for some
$f\in C(K,\C^N)$ and $\gamma\in\Gamma(G\times K, U(N))$. 
It generalizes Theorem~8.1 of Lenz \cite{lenz09a} to amenable semigroups. 

\begin{thm}
  \label{thm:mean-ergodic-koopman-in-f}
Let the action of a (right) amenable semigroup $G$ on $K$ be uniquely
ergodic with invariant measure $\mu$ and let $\S$ and $\S_2$ be the
  corresponding Koopman representations on $C(K,\C^N)$ and $L^2(K,\C^N,\mu)$, respectively. 
If $\gamma: G\times K\to U(N)$ is a continuous cocycle and $f\in
C(K,\C^N)$ is given, then the following assertions are equivalent.
\begin{enumerate}[(1)]
\item $P_{\Fix\gamma\S_2}f\in \Fix\gamma\S$.
\item $\gamma\S$ is mean ergodic on $Y_f$ with mean ergodic projection $P_\gamma$.
\item $\Fix\gamma\S|_{Y_f}$ separates $\Fix(\gamma\S)|_{Y_f}'$.
\item $f\in \Fix\gamma\S\oplus\ol{\lin}\rg(I-\gamma\S)$.
\item $A_\a^{\gamma\S} f$ converges weakly (to a fixed point of
  $\gamma\S$) for some/every strong (right) $\gamma\S$-ergodic
  net $(A_\a^{\gamma\S})$.
\item $A_\a^{\gamma\S} f$ converges strongly (to a fixed point of
  $\gamma\S$) for some/every strong (right) $\gamma\S$-ergodic net
  $(A_\a^{\gamma\S})$. 
\end{enumerate}
The limit $P_\gamma$ of the nets $A_\a^{\gamma\S}$ in the strong (weak, resp.) operator
topology on $Y_f$ is the mean ergodic projection of
$\gamma\S|_{Y_f}$ mapping $Y_f$ onto $\Fix\gamma\S|_{Y_f}$ along
$\ol{\lin}\rg(I-\gamma\S|_{Y_f})$.
\end{thm} 
\begin{proof}
  The equivalence of the statements (2) to (6) follows directly from
\cite[Proposition 1.11]{schreiber12}.

(6)$\imp$(1): By von Neumann's Ergodic Theorem $P_{\Fix\gamma\S_2}f$
is the limit of $A_\a^{\gamma\S} f$ in $L^2(K,\C^N,\mu)$. 
If $A_\a^{\gamma\S} f$ converges strongly in $C(K,\C^N)$ then the limits
coincide almost everywhere and hence $P_{\Fix\gamma\S_2}f$ has a
continuous representative in $\Fix\gamma\S$.

(1)$\imp$(4): Let $\nu\in C(K,\C^N)'$ vanish on
$\Fix\gamma\S\oplus\ol{\lin}\rg(I-\gamma\S)$. Then in particular
$\sk{\nu,h}=\sk{\nu,\gamma(g,\cdot)S_g h}=\sk{(\gamma(g,\cdot)S_g)' \nu, h}$ for all
$h\in C(K,\C^N)$ and $g\in G$ and thus $\nu\in \Fix(\gamma\S)'$. Hence by Lemma
\ref{lemma:Fixraum_dualer_Fixraum_cocycles} there exists $h\in
\Fix(\gamma\S_2)^*$ such that $\nu=\ol{h} d\mu$. 
Let $(A_\a^{\gamma\S_2})_{\a\in\mathcal{A}}$ be a strong
$\gamma\S_2$-ergodic net on $L^2(K,\C^N\mu)$. 
Then $(A_\a^{\gamma\S_2})^* h=h$ for all $\a\in\mathcal{A}$ and von
Neumann's Ergodic Theorem implies
$$\sk{\nu,f}=\sk{f,h}_{L^2}=\sk{A_\a^{\gamma\S_2}
  f,h}_{L^2}\to \langle  \underbrace{P_{\Fix\gamma\S_2}f}_{\in
    C(K,\C^N)},h\rangle_{L^2}= \langle\nu,\underbrace{P_{\Fix\gamma\S_2}f}_{\in
    \Fix\gamma\S}\rangle=0.$$
Hence the Hahn-Banach Theorem yields $f\in
\Fix\gamma\S\oplus\ol{\lin}\rg(I-\gamma\S)$, since
$\Fix\gamma\S\oplus\ol{\lin}\rg(I-\gamma\S)$ is closed by Theorem~1.9 in
Krengel \cite[Chap.~2]{krengel85}.
\end{proof}

\begin{remark}
\label{rem:fixraum}
  Notice that if $P_{\Fix\gamma\S_2}f=0$ in the situation of
  Theorem~\ref{thm:mean-ergodic-koopman-in-f} then $\gamma\S$ is mean
  ergodic on $Y_f$ with mean ergodic projection $P_\gamma=0$. 
  This observation
  then directly implies the notable fact that if
  $\Fix\gamma\S_2=\{0\}$, then $\Fix\gamma\S=\{0\}$.
\end{remark}

Analogously to Corollary~\ref{cor:uniform-convergence_C_K} we consider
the semigroups $\gamma\S$ for $\gamma\in\Lambda\subseteq \Gamma(G\times
K, Z)$, where $Z$ is a compact subgroup of $U(N)$, and ask when
$A_\a^{\gamma\S}f$ converges uniformly in $\gamma\in\Lambda$ for a
uniform family of right ergodic nets $\{(A_\a^{\gamma\S})_{\a\in \mathcal{A}}:
\gamma\in \Lambda\}$ on $C(K,\C^N)$ and a given $f\in C(K,\C^N)$. 
If $\gamma\S$ is a mean ergodic semigroup on $\ol{\lin}\gamma\S f$, we denote by
$P_\gamma$ its mean ergodic projection.
The following corollary is a cocycle version of Theorem~2 in
\cite{lenz09} for amenable semigroups.

\begin{cor}
  \label{cor:uniform-convergence_Walters}
Let the action of a right amenable semigroup $G$ on $K$ be uniquely
ergodic with invariant measure $\mu$ and let $\S$ and $\S_2$ be the
  corresponding Koopman representations on $C(K,\C^N)$ and $L^2(K,\C^N,\mu)$, respectively. 
Assume that $\Lambda\subseteq
\Gamma(G\times K, Z)$ is compact and consider the semigroups
$\gamma\S$ on $C(K,\C^N)$ for each $\gamma\in\Lambda$. If
$\{(A_\a^{\gamma\S})_{\a\in \mathcal{A}}: \gamma\in \Lambda\}$ is a uniform
family of right ergodic nets on $C(K,\C^N)$ and if $f\in C(K,\C^N)$ satisfies
\begin{enumerate}[(1)]
\item $P_{\Fix\gamma\S_2}f\in \Fix\gamma\S$ for each $\gamma\in \Lambda$ and
\item $\Lambda \to \R_+, \gamma\mapsto\|A_\a^{\gamma\S}f-
  P_\gamma f\|$ is continuous for each $\a\in\mathcal{A}$,
\end{enumerate}
then 
$$\lim_\a\sup_{\gamma\in \Lambda}\|A_\a^{\gamma\S}f-P_\gamma f\|= 0.$$ 
\end{cor}
\begin{proof}
By Theorem~\ref{thm:mean-ergodic-koopman-in-f} and our
 hypotheses the semigroup $\gamma\S$ is mean ergodic on $\ol{\lin}\gamma\S f$ for each
 $\gamma\in\Lambda$.
The result then follows directly from Theorem~2.4 in \cite{schreiber12}.
\end{proof}

In order to show the analogue of Corollary~\ref{cor:assani} for cocycles we need a lemma.

\begin{lemma}
  \label{lemma:uniform-family}
 Let $H$ be a locally compact group with left Haar measure $|\cdot|$ and
  suppose that $G\subset H$ is a subsemigroup acting on $K$ such that there exists a F\o
  lner net $(F_\a)_{\a\in\A}$ in $G$. 
Consider the semigroups $\gamma\S$ on $C(K,\C^N)$ for each $\gamma$ in
a compact set $\Lambda\subset \Gamma(G\times K, Z)$.
Then $\{(A_\a^{\gamma\S})_{\a\in\A}: \gamma\in\Lambda\}$ defined by
$$A_\a^{\gamma\S}f:=\frac{1}{|F_\a|}\int_{F_\a}\gamma(g,\cdot)S_g
f\,dg\quad (f\in C(K,\C^N))$$
is a uniform family of right ergodic nets.
\end{lemma}
\begin{proof}
(1): Let $\a\in \A$, $\e>0$ and $f_1,\dots, f_m\in C(K,\C^N)$. 
Since $\Lambda\subset \Gamma(G\times K, Z)$ is compact the family $\{g\mapsto \gamma(g,\cdot)S_gf_k:
  \gamma\in\Lambda\}$ is uniformly equicontinuous on the compact set $F_\a$ for each $k\in\{1,\dots, m\}$. Hence for each $k\in\{1,\dots, m\}$ we can
choose an open neighbourhood $U_k$ of the unity of $H$
  satisfying 
$$g,h\in G, \en h^{-1}g\in U_k \en \imp\en \sup_{\gamma\in
  \Lambda}\|\gamma(g,\cdot)S_g f_k-\gamma(h,\cdot)S_h f_k\|<\e.$$
Then $U:=\bigcap_{k=1}^m U_k$ is still an open neighbourhood of
unity. Since $F_\a$ is compact there exists $g_1,\dots, g_n\in F_\a$
such that $F_\a\subset\bigcup_{j=1}^n g_j U$. Defining $V_1:=g_1 U\cap
F_\a$ and $V_j:=(g_j U\cap F_\a) \setminus V_{j-1}$ for $j=2,\dots, n$
we obtain a disjoint union $F_\a=\bigcup_{j=1}^n V_j.$
Hence for all $\gamma\in \Lambda$ and $k\in \{1,\dots, m\}$ we have
\begin{align*}
  \left\| \frac{1}{|F_\a|}\int_{F_\a}\gamma(g,\cdot)S_g f_k\,
    dg-\right . &\left .\sum_{j=1}^n \frac{|V_j|}{|F_\a|} \gamma(g_j,\cdot) S_{g_j}f_k
  \right\|\\
&\le \frac{1}{|F_\a|}\sum_{j=1}^n
  \int_{V_j}\underbrace{\|\gamma(g,\cdot)S_g f_k-\gamma(g_j,\cdot)S_{g_j} f_k\|}_{<\e}dg\\
&< \frac{1}{|F_\a|}\sum_{j=1}^n |V_j|\e=\e.
\end{align*}

(2): Let $g\in G$ and $f\in C(K,\C^N)$. Then 
$$\|\gamma(g,\cdot)S_g f\|=\sup_{x\in K}\|\gamma(g,x)f(gx)\|_2 =\sup_{x\in
  K}\|f(gx)\|_2\le \|f\|,$$ since $\gamma(g,x)$ is unitary for all
$x\in K$. Hence
\begin{align*}
\sup_{\gamma\in\Lambda}\left\|\frac{1}{|F_\a|}\int_{F_\a}\gamma(g,\cdot))S_gf-\gamma(hg,\cdot)S_{hg}f\,dg\right\|&
  \le  \sup_{\gamma\in\Lambda} \frac{1}{|F_\a|}\int_{F_\a\triangle h F_\a} \|\gamma(g,\cdot) S_g f\|\\
&\le \frac{|F_\a\triangle h F_\a|}{|F_\a|}  \|f\|\to 0.
\end{align*}
\end{proof}

For actions of the semigroup $(\N,+)$ and the F\o lner sequence
$F_n=\{0,1,\dots, n-1\}$ the following corollary 
has been proved by Santos and Walkden \cite[Corollary~4.4]{santos07}
generalizing a previous result of Walters \cite[Theorem~5]{walters96},
who considered the case $N=1$. 
\begin{cor}
\label{cor:santos-walkden}
  Let $H$ be a locally compact group with left Haar measure $|\cdot|$ and
  suppose that $G\subset H$ is a subsemigroup such that there exists a F\o
  lner net $(F_\a)_{\a\in\A}$ in $G$. If the action of $G$ on $K$ is
  uniquely ergodic with invariant measure $\mu$ and if $f\in
  C(K,\C^N)$ satisfies $P_{\Fix\gamma\S_2}f=0$ for 
  all $\gamma$ in a compact set $\Lambda\subset\Gamma(G\times K, Z)$, then 
$$\lim_{\a}\sup_{\gamma\in\Lambda}\left\|\frac{1}{|F_\a|}\int_{F_\a}\gamma(g,\cdot)S_g
  f \;dg\right\|=0.$$
\end{cor}
\begin{proof}
   If $(F_\a)$ is a F\o lner net in $G$, then $G$ is right
   amenable. By Lemma~\ref{lemma:uniform-family} the set
$\{(A_\a^{\gamma\S})_{\a\in\A}: \gamma\in\Lambda\}$ defined by
$$A_\a^{\gamma\S}\psi:=\frac{1}{|F_\a|}\int_{F_\a}\gamma(g,\cdot)S_g
\psi\,dg\quad (\psi\in C(K,\C^N))$$
is a uniform family of right ergodic nets on $C(K,\C^N)$.
If $P_{\Fix\gamma\S_2}f=0$ for all
$\gamma\in\Lambda$, then the conditions (1) and (2) of
Corollary~\ref{cor:uniform-convergence_Walters} are satisfied since
the map 
$\gamma\mapsto \frac{1}{|F_\a|}\int_{F_\a}\gamma(g,\cdot)S_g
  f dg$ is continuous. Hence 
$$\lim_{\a}\sup_{\gamma\in\Lambda}\left\|\frac{1}{|F_\a|}\int_{F_\a}\gamma(g,\cdot)S_g
  f \;dg\right\|=0.$$
\end{proof}

\section{Mean ergodicity on group extensions}

In this section we characterize mean ergodicity of semigroups of Koopman
operators associated to skew product actions on compact
group extensions.


Let $G$ be an amenable semigroup acting on a compact space $K$ and assume that $\mu$ is a
$G$-invariant probability measure on $K$. Suppose that $\Omega$ is a
compact group with Haar measure $\eta$ 
and $\gamma:G\times K\to\Omega$ is a
continuous cocycle. We define the \emph{skew product action} of $G$ on
$K\times \Omega$ by 
$$g(x,\omega):=(gx, \omega \gamma(g,x)),\quad (x,\omega)\in
 K\times\Omega, g\in G.$$
The cocycle equation implies that this is indeed a semigroup action
and one checks that the product measure $\mu\times\eta$ on $K\times\Omega$ 
is $G$-invariant. 
We denote by $\T=\{T_g: g\in G\}$ the Koopman
representation of this action on $C(K\times\Omega)$ and by
$\T_2=\{T_{g,2}: g\in G\}$ its extension to $L^2(K\times\Omega,\mu\times \eta)$.

In order to study the mean ergodicity of $\T$, we need some harmonic analysis. 
Denote by $\widehat{\Omega}$ the set of
irreducible representations of $\Omega$ and by $\Rep(\Omega,N)$ the set
of irreducible $N$-dimensional representations of $\Omega$ (see e.g. Folland
\cite[Chap.~3.1]{folland95}). 
Notice that $\widehat{\Omega}=\bigcup_{N\in\N}\Rep(\Omega,N)$ by
\cite[Theorem~7.2.4]{deitmar09}. If $\pi\in \Rep(\Omega,N)$, we can
choose an inner product on $\C^N$ such  that $\pi$ becomes unitary
(see \cite[Lemma 7.1.1]{deitmar09}). 
Hence, we may
always assume that each finite dimensional representation is unitary.
 For $\pi\in\Rep(\Omega,N)$ and for a fixed orthonormal basis $\{e_1,\dots, e_N\}$ in $\C^N$ the maps
$\pi_{i,j}\in C(\Omega)$ defined by 
$\pi_{i,j}(\omega)=\sk{\pi(\omega)e_i, e_j}$ 
are called the \emph{matrix elements} of $\pi$. 

On $C(K,\C^N)$ we consider the Koopman representation $\S^{(N)}=\{S_g: g\in
G\}$  of $G$. 
If $\gamma:G\times K\to\Omega$ is a continuous cocycle and $\pi\in
\Rep(\Omega, N)$ is an $N$-dimensional representation of $G$ then
$\pi\circ\gamma: G\times K\to U(N)$ is a continuous cocycle into the
group of unitary operators on $\C^N$. 
Hence, in accordance with the notation from Section~\ref{sec:koopman}
for $g\in G$ the operator
$(\pi\circ\gamma)(g,\cdot) S_g$ is defined by 
$$(\pi\circ\gamma)(g,\cdot) S_g f(x)=\pi(\gamma(g,x))f(gx),\quad f\in
C(K,\C^N), x\in K.$$
We denote by $(\pi\circ\gamma)\S^{(N)}:=\{(\pi\circ\gamma)(g,\cdot)S_g: g\in G\}$ the
corresponding semigroup and by $(\pi\circ\gamma)\S_2^{(N)}$ its
extension to $L^2(K,\C^N,\mu)$.

\begin{thm}
 \label{thm:mean-ergodic-group-extension}
Let $G$ be a right amenable semigroup acting on $K$ and suppose that $\Omega$ is a
compact group and $\gamma:G\times K\to\Omega$  a continuous
cocycle. 
For the Koopman representations $\S^{(N)}$ of $G$ on $C(K,\C^N)$ and the
Koopman representation $\T$ of the skew product action on
$C(K\times\Omega)$ the following assertions are equivalent.
\begin{enumerate}[(1)]
\item $\T$ is mean ergodic on $C(K\times\Omega)$.
\item $(\pi\circ \gamma)\S^{(N)}$ is mean ergodic on
  $C(K,\C^N)$ for each $\pi\in\Rep(\Omega,N)$ and each $N\in\N$.
\end{enumerate}
\end{thm}
\begin{proof}
For a fixed
orthonormal basis $\{e_1,\dots, e_N\}$ of $\C^N$ every $f\in
C(K,\C^N)$ can be written as $f=\sum_{i=1}^N f_i e_i$ for functions
$f_i\in C(K)$. 
Therefore, for each $\pi\in\Rep(\Omega,N), (x,\omega)\in K\times \Omega$ and
$g\in G$ we have 
\begin{align*}
  \sum_{i=1}^N\sum_{j=1}^N T_{g}(f_i\otimes
\pi_{ij})(x,\omega)e_j&= \sum_{i=1}^N\sum_{j=1}^N f_i(gx)
\langle\pi(\omega\gamma(g,x))e_i,e_j\rangle e_j\\
&= \sum_{i=1}^N \pi(\omega)\pi(\gamma(g,x)) f_i(gx)e_i\\
&=\pi(\omega)\pi(\gamma(g,x))S_gf (x).
\end{align*}

(1)$\imp$(2): Take $P\in\ol{\co}\T$ such that $T_gP=PT_g=P$ for all
$g\in G$ and $\pi\in\Rep(\Omega,N)$ for some $N\in\N$. 
Define the operator $Q$ on $C(K,\C^N)$ by 
$$Qf(x):= \sum_{i=1}^N\sum_{j=1}^N P(f_i\otimes
\pi_{ij})(x,1_\Omega)e_j,\quad (f\in C(K,\C^N), x\in K)$$
where $1_\Omega$ is the unit element of $\Omega$. 
We claim that $Q$ is the mean ergodic projection of $(\pi\circ \gamma)\S^{(N)}$.
Since $P\in\ol{\co}\T$, there exists a net
$(\sum_{n=1}^{N_\a}\lambda_{n,\a}T_{g_n})_\a\subseteq\co\T$ with $PF=\lim_\a \sum_{n=1}^{N_\a}\lambda_{n,\a}T_{g_n}F$ for all
$F\in C(K\times\Omega)$.   
For every $f\in C(K,\C^N)$ we thus obtain 
\begin{align*}
  Qf(x)&=\lim_{\a}\sum_{n=1}^{N_\a}\lambda_{n,\a}\sum_{i=1}^N\sum_{j=1}^N T_{g_n}(f_i\otimes
\pi_{ij})(x,1_\Omega)e_j\\
&=\lim_{\a}\sum_{n=1}^{N_\a}\lambda_{n,\a}\pi(1_\Omega)\pi(\gamma(g_n,x))S_{g_n}f(x) ,
\end{align*}
where the limit is uniform in $x\in K$. This yields
$$Qf=\lim_{\a}\sum_{n=1}^{N_\a}\lambda_{n,\a}(\pi\circ\gamma)(g_n,\cdot)S_{g_n}f$$ 
for all $f\in C(K,\C^N)$ and hence $Q\in\ol{\co}(\pi\circ\gamma)\S^{(N)}$. 

To see that $Q$ is a null element of $\ol{\co}(\pi\circ\gamma)\S^{(N)}$
let $g\in G$ and $f\in C(K,\C^N)$. We then have
\begin{align*}
  Q((\pi\circ\gamma)(g,\cdot) S_gf) 
 &=\lim_{\a}\sum_{n=1}^{N_\a}\lambda_{n,\a}((\pi\circ\gamma)(g_k,\cdot)S_{g_k})
 (\pi\circ\gamma)(g,\cdot) S_gf\\
&=\lim_{\a}\sum_{n=1}^{N_\a}\lambda_{n,\a}(\pi\circ\gamma)(gg_k,\cdot)S_{gg_k}f\\
&=\lim_{\a}\sum_{n=1}^{N_\a}\lambda_{n,\a}\sum_{i=1}^N\sum_{j=1}^NT_{gg_k}(f_i\otimes\pi_{ij})(\cdot,1_\Omega)e_j\\
&=\sum_{i=1}^N\sum_{j=1}^N \underbrace{PT_g}_{=P}
(f_i\otimes\pi_{ij})(\cdot,1_\Omega)e_j\\
&=Qf.
\end{align*}
Analogously, one verifies $(\pi\circ\gamma)(g,\cdot) S_g Qf=Qf$ and
hence $Q$ is a null element of $\ol{\co}(\pi\circ\gamma)\S^{(N)}$.

(2)$\imp$(1):
Take a strong right $\T$-ergodic net $(A_\a)_{\a\in\mathcal{A}}$ on
$C(K\times\Omega)$ with $A_\a\in \co\T$ for
all $\a\in\mathcal{A}$ (cf. the proof of
Proposition~1.3 and
Theorem~1.4 in \cite{schreiber12}). 
So suppose that $A_\a=\sum_{n=1}^{N_\a}\lambda_{n,\a}T_{g_n}\in\co\T$
for all $\a\in\mathcal{A}$. 
For $\T$ to be mean ergodic we need to show that $(A_\a F)$ converges
strongly to a fixed point of $\T$ for every $F\in C(K\times\Omega)$. 
By \cite[Theorem~5.11]{folland95} the linear span of the matrix elements is
dense in $C(\Omega)$. Since $C(K)\otimes C(\Omega)$ is dense in
$C(K\times\Omega)$ and the $A_\a$ are linear and uniformly bounded, it
thus suffices to show that $A_\a(f\otimes\pi_{ij})$ converges for
every $f\in C(K)$, $\pi\in\Rep(\Omega,N)$ and $i,j\in\{1,\dots,N\}$ for each $N\in\N$.
So take $f\in C(K)$, $\pi\in\Rep(\Omega,N)$ and fix an orthonormal
basis $\{e_1,\dots,e_N\}$ of $\C^N$. 
For each $(x,\omega)\in K\times\Omega$ and $i,j\in\{1,\dots, N\}$ we have
\begin{align*}
  A_\a(f\otimes\pi_{ij})(x,\omega)&=\sum_{n=1}^{N_\a}\lambda_{n,\a}(f\otimes\pi_{ij})(g_n
  x,\omega\gamma(g_n,x))\\
&=\sum_{n=1}^{N_\a}\lambda_{n,\a}f(g_n x)\sk{\pi(\omega)\pi(\gamma(g_n,x))e_i,e_j}\\
&=\left\langle\pi(\omega) \underbrace{\sum_{n=1}^{N_\a}\lambda_{n,\a}(\pi(\gamma(g_n,\cdot))S_{g_n})}_{=:B_\a}f^{(i)}(x),e_j\right\rangle,
\end{align*}
where $f^{(i)}\in C(K,\C^N)$ is defined by $f^{(i)}(x)=f(x)e_i$. One verifies
that the net $(B_\a)_{\a\in \mathcal{A}}$ forms a strong right
$(\pi\circ\gamma)\S^{(N)}$-ergodic net. 
Since by hypothesis 
$(\pi\circ\gamma)\S^{(N)}$ is mean ergodic, it follows from
Theorem~\ref{thm:mean-ergodic-koopman} that
$B_\a f^{(i)}$ converges in $C(K,\C^N)$ to a fixed point $h^{(i)}$ of
$(\pi\circ\gamma)\S^{(N)}$ for each $i=1,\dots,N$. 
Defining the function $h_{ij}\in C(K\times\Omega)$ by $h_{ij}(x,\omega):=\sk{\pi(\omega) h^{(i)}(x),e_j}$, we thus obtain
\begin{align*}
  \|A_\a(f\otimes\pi_{ij})&-h_{ij}\|_{C(K\times\Omega)}\\
&=\sup_{(x,\omega)\in K\times \Omega}
\left|\sk{\pi(\omega)\left(\sum_{n=1}^{N_\a}
\lambda_{n,\a}(\pi(\gamma(g_n,\cdot))S_{g_n})f^{(i)}(x)-h^{(i)}(x)\right),e_j}\right|\\
&\le
\left\|\sum_{n=1}^{N_\a}\lambda_{n,\a}(\pi(\gamma(g_n,\cdot))S_{g_n})f^{(i)}-h^{(i)}\right\|_{C(K,\C^N)} 
\to 0,
\end{align*}
since $\pi(\omega)$ is unitary for each $\omega\in\Omega$.
Hence the net $A_\a(f\otimes\pi_{ij})$ converges in $C(K\times\Omega)$
to the fixed point $h_{ij}$ of $\T$. 
\end{proof}

In Theorem~\ref{thm:mean-ergodic-group-extension} we have reduced the
mean ergodicity of $\T$ to the mean ergodicity of the semigroups
$(\pi\circ\gamma)\S^{(N)}$ on $C(K,\C^N)$. 
Theorem~\ref{thm:mean-ergodic-koopman} now provides a useful criterion
for the mean ergodicity of the semigroup
$(\pi\circ\gamma)\S^{(N)}$. 
Hence, combining these two results we obtain the following corollary.


\begin{cor}
  \label{cor:1}
Let the action of a right amenable semigroup $G$ on $K$ be uniquely
ergodic with  invariant measure $\mu$ and suppose that $\Omega$ is a compact group
and $\gamma:G\times K\to\Omega$ a continuous cocycle. 
For the Koopman representations $\S^{(N)}$ and $\S_2^{(N)}$ of $G$ on
$C(K,\C^N)$ and $L^2(K,\C^N,\mu)$, respectively, and
the Koopman representation $\T$ of the skew product action on
$C(K\times\Omega)$ the following assertions are equivalent.
  \begin{enumerate}[(1)]
  \item $\T$ is mean ergodic on $C(K\times\Omega)$.
  \item $\Fix(\pi\circ\gamma)\S_2^{(N)}\subseteq \Fix(\pi\circ\gamma)\S^{(N)}$ for each
    $\pi\in\Rep(\Omega,N)$ and each $N\in\N$.
  \end{enumerate}
\end{cor}





For the proof of Theorem~\ref{thm:furstenberg} below we need a characterization of the
ergodicity of the above skew product action. 
For this purpose we extend Theorem~2.1 of Keynes and Newton
\cite{keynes-newton76} to semigroup actions.  

The trivial $1$-dimensional representation $\omega\mapsto 1$ on $\Omega$ is denoted by $\1$. 
\begin{prop}
  \label{prop:ergodicity}
Let the action of a semitopological semigroup $G$ on $K$ be ergodic
with respect to some invariant measure $\mu$ and suppose that $\Omega$ is a
compact group with Haar measure $\eta$ and $\gamma:G\times K\to\Omega$
a continuous cocycle. 
For the Koopman representations $\S_2^{(N)}$ of $G$ on $L^2(K,\C^N,\mu)$ and
the Koopman representation $\T_2$ of the skew product action on
$L^2(K\times\Omega,\mu\times \eta)$ the following assertions are
equivalent. 
\begin{enumerate}[(1)]
\item $\Fix\T_2=\C\cdot\1$.
\item $\Fix (\pi\circ\gamma)\S_2^{(N)}=\{0\}$ for each $\pi\in
  \Rep(\Omega,N)\setminus \{\1\}$ and each $N\in\N$.
\end{enumerate}
\end{prop}

\begin{proof}
  (1)$\imp$(2): Assume (1) and suppose there exists $0\neq
  f\in\Fix(\pi\circ\gamma)\S_2^{(N)}$ for some $N$-dimensional representation $\pi\in
  \Rep(\Omega,N)\setminus \{\1\}$. 
  Fix an orthonormal basis
  $\{e_1,\dots, e_N\}$ of $\C^N$ and define $F_j\in
  L^2(K\times\Omega,\mu\times \eta)$ by
  $F_j(x,\omega)=\sk{\pi(\omega)f(x),e_j}$ for each
  $j\in\{1,\dots,N\}$. Then for each $j\in\{1,\dots, N\}$, $g\in G$ and $\mu\times \eta$-a.e.
  $(x,\omega)\in K\times\Omega$ we have 
  \begin{align*}
T_{g,2}F_j(x,\omega)&= F_j(gx,\omega\gamma(g,x))
=\sk{\pi(\omega)\pi(\gamma(g,x))f(gx),e_j}\\
&=\sk{\pi(\omega)f(x), e_j}=F_j(x,\omega).
  \end{align*}
Hence  $F_j\in \Fix\T_2$ for all $j\in\{1,\dots,N\}$ and thus each $F_j$ is constant. Since $f\in L^2(K, \C^N,\mu)$, we can
write $f$ as $f=\sum_{i=1}^N f_i e_i$ for $f_i\in L^2(K,\mu)$. By Fubini's Theorem we then obtain
\begin{align*}
  \int_{K\times\Omega}F_j(x,\omega)\; d(\mu\times \eta)(x,\omega)
&=
\sum_{i=1}^N\int_{K\times\Omega}f_i(x)\sk{\pi(\omega)e_i,e_j}\;d(\mu\times
\eta)(x,\omega)\\
&=\sum_{i=1}^N\int_K f_i(x)\; d\mu(x)\underbrace{\int_\Omega \pi_{ij}(\omega)\; d\eta(\omega)}_{\sk{\pi_{ij},\1}_{L^2(\Omega,\eta)}}=0
\end{align*}
for all $j\in\{1,\dots,N\}$, since $\pi\neq\1$ (see \cite[Theorem~7.2.1]{deitmar09}). Hence
$F_j=0$ for all $j\in\{1,\dots,N\}$ and thus $f=0$, which contradicts
the assumption.

(2)$\imp$(1): Let $F\in\Fix \T_2$, $N\in\N$, $\pi\in\Rep(\Omega,N)$, $i\in\{1,\dots,N\}$  and define 
$$h_{\pi,i}(x)=\int_\Omega
F(x,\omega)\pi(\omega)^{-1}e_i\;d\eta(\omega).$$ 
Then for every $g\in G$ and $\mu$-a.e. $x\in K$ we have 
\begin{align*}
  (\pi\circ\gamma)(g,\cdot)S_{g,2}
  h_{\pi,i}(x)&=\pi(\gamma(g,x))\int_\Omega
F(gx,\omega)\pi(\omega)^{-1}e_i\;d\eta(\omega)\\
&=\pi(\gamma(g,x))\int_\Omega
F(gx,\omega\gamma(g,x))\pi(\gamma(g,x))^{-1}\pi(\omega)^{-1}e_i\;d\eta(\omega)\\
&=\int_\Omega F(x,\omega)\pi(\omega)^{-1}e_i\;d\eta(\omega)=h_{\pi,i}(x)
\end{align*}
by the invariance of the Haar measure $\eta$.
Hence $h_{\pi,i}\in\Fix(\pi\circ\gamma)\S_2=\{0\}$ and thus
\begin{align*}
  0&=\int_\Omega \sk{F(x,\omega)\pi(\omega)^{-1}e_i,e_j}d\eta(\omega)\\
&=\int_\Omega F(x,\omega)\ol{\pi_{ji}(\omega)}d\eta(\omega)\\
&=\sk{F(x,\cdot),\pi_{ji}}_{L^2(\Omega,\eta)}
\end{align*}
for each $\pi\in\Rep(\Omega,N)\setminus\{\1\}$, each
$i,j\in\{1,\dots,N\}$ and $\mu$-a.e. $x\in K$. 
Hence $F(x,\cdot)$ is constant for $\mu$-almost every
$x\in K$ and thus there exists $f\in L^2(K,\mu)$ with $F=f\otimes
\1$. Since $F\in\Fix\T_2$ it follows that $f\in\Fix\S_2^{(1)}$ and by
ergodicity $f$ and consequently $F$ is constant. 
\end{proof}

 The following result is due to Furstenberg
  \cite[Proposition~3.10]{furstenberg81} in the case of an $\N$-action
  on a compact  metric space $K$. 
Our proof does not use the Pointwise Ergodic Theorem and so-called generic points.

\begin{thm}
  \label{thm:furstenberg}
Let the action of a right amenable semigroup $G$ on $K$ be uniquely
ergodic with invariant measure $\mu$ and suppose that $\Omega$ is a compact group
with Haar measure $\eta$ and $\gamma:G\times K\to\Omega$ a continuous cocycle.
If the skew product action 
$(g,(x,\omega))\mapsto (gx,\omega\gamma(g,x))$ of $G$ on $K\times\Omega$ is
ergodic with respect to the product measure $\mu\times \eta$, then it
is uniquely ergodic. 
\end{thm}
\begin{proof}
  The ergodicity of the skew product action is equivalent to $\Fix\T_2=\C\cdot\1$ and thus
  $\Fix(\pi\circ\gamma)\S_2^{(N)}=\{0\}$ for each
  $\pi\in\Rep(\Omega,N)\setminus \{\1\}$ and $N\in\N$ by
  Proposition~\ref{prop:ergodicity}. 
 Since $\Fix\S_2^{(1)}=\C\cdot \1$ on $L^2(K,\mu)$ by ergodicity this yields
 $\Fix(\pi\circ\gamma)\S_2^{(N)}\subseteq \Fix(\pi\circ\gamma)\S^{(N)}$ for each  
  $\pi\in\Rep(\Omega,N)$ and $N\in\N$. 
Hence $\T$ is mean ergodic by Corollary
  \ref{cor:1}.
To show that $\T$ is uniquely ergodic by
  Proposition~\ref{prop:uniquely-ergodic-m-erg} it thus remains to verify that $\Fix\T=\C\cdot\1$.
Take a strong right $\T$-ergodic net $(A_\a)_{\a\in\mathcal{A}}$ on
$C(K\times\Omega)$ with  $A_\a=\sum_{n=1}^{N_\a}\lambda_{n,\a}T_{g_n}\in\co\T$
for all $\a\in\mathcal{A}$. 
Since $\T$ is mean ergodic the net $(A_\a)$ converges strongly to a projection $P$ with $\ran P=\Fix\T$.
By density of the linear span of the matrix elements in $C(\Omega)$ it thus suffices to show that $P(f\otimes\pi_{ij})$ is constant for each $f\in C(K)$ and each matrix element $\pi_{ij}$.

So take $f\in C(K)$, $\pi\in\Rep(\Omega,N)$ and fix an orthonormal
basis $\{e_1,\dots,e_N\}$ of $\C^N$. 
For each $(x,\omega)\in K\times\Omega$ and $i,j\in\{1,\dots, N\}$ we have
\begin{align*}
  A_\a(f\otimes\pi_{ij})(x,\omega)&=\sum_{n=1}^{N_\a}\lambda_{n,\a}(f\otimes\pi_{ij})(g_nx,\omega\gamma(g_n,x))\\
&=\sum_{n=1}^{N_\a}\lambda_{n,\a}f(g_n x)\sk{\pi(\omega)\pi(\gamma(g_n,x))e_i,e_j}\\
&=\left\langle\pi(\omega) \underbrace{\sum_{n=1}^{N_\a}\lambda_{n,\a}(\pi(\gamma(g_n,\cdot))S_{g_n})}_{=:B^{\pi}_\a}f^{(i)}(x),e_j\right\rangle,
\end{align*}
where $f^{(i)}\in C(K,\C^N)$ is defined by $f^{(i)}(x)=f(x)e_i$. 
One verifies
that the net $(B^{\pi}_\a)_{\a\in \mathcal{A}}$ forms a strong right
$(\pi\circ\gamma)\S^{(N)}$-ergodic net. 
From Remark~\ref{rem:fixraum} we deduce $\Fix(\pi\circ\gamma)\S^{(N)}=\{0\}$ for each $\pi\in\Rep(\Omega,N)\setminus\{\1\}$ and $N\in\N$, hence $B^{\pi}_\a f^{(i)}\to 0$ in $C(K,\C^N)$ for each $\pi\in\Rep(\Omega,N)\setminus\{\1\}$ and $N\in\N$.
For $\pi=\1$ the net $B_\a f^{(1)}$ converges to a constant $c\cdot\1$ since $\Fix\S^{(1)}=\C\cdot\1$ by unique ergodicity.
Hence if $\pi\neq\1$ we obtain
\begin{align*}
 \| A_\a(f\otimes\pi_{ij})\|_{C(K\times\Omega)}&=\sup_{(x,\omega)\in K\times \Omega}
\left|\sk{\pi(\omega)B^{\pi}_\a f^{(i)}(x),e_j}\right|\\
&\le
\left\|B^{\pi}_\a f^{(i)}\right\|_{C(K,\C^N)} 
\to 0,
\end{align*}
since $\pi(\omega)$ is unitary for each $\omega\in\Omega$.
This yields $P(f\otimes\pi_{ij})=0$ if $\pi\neq\1$ and the same calculation shows
 $P(f\otimes\1)=c\cdot\1$ if $\pi=\1$.
Hence $\Fix\T=\ran P=\C\cdot\1$, which finishes the proof.
\end{proof}

\textbf{Acknowledgement.} The author is grateful to Rainer Nagel and
Pavel Zorin-Kranich for valuable comments and interesting discussions.


\bibliographystyle{siam}
\bibliographystyle{amsplain}

\end{document}